\documentclass[12pt]{amsart}
\usepackage{amssymb,amscd,amsmath}
\usepackage{verbatim,color}
\usepackage{graphics}
\usepackage{url}

\usepackage{graphicx}

\usepackage[lined,algonl,boxed,norelsize]{algorithm2e}

\begin{document}


\numberwithin{equation}{section}

\newtheorem{theorem}[equation]{Theorem}
\newtheorem{lemma}[equation]{Lemma}
\newtheorem{conjecture}[equation]{Conjecture}
\newtheorem{proposition}[equation]{Proposition}
\newtheorem{corollary}[equation]{Corollary}
\newtheorem{cor}[equation]{Corollary}

\theoremstyle{definition}
\newtheorem*{definition}{Definition}
\newtheorem{example}[equation]{Example}

\theoremstyle{remark}
\newtheorem{remark}[equation]{Remark}
\newtheorem{remarks}[equation]{Remarks}
\newtheorem*{acknowledgement}{Acknowledgements}


\newenvironment{notation}[0]{%
  \begin{list}%
    {}%
    {\setlength{\itemindent}{0pt}
     \setlength{\labelwidth}{4\parindent}
     \setlength{\labelsep}{\parindent}
     \setlength{\leftmargin}{5\parindent}
     \setlength{\itemsep}{0pt}
     }%
   }%
  {\end{list}}

\newenvironment{parts}[0]{%
  \begin{list}{}%
    {\setlength{\itemindent}{0pt}
     \setlength{\labelwidth}{1.5\parindent}
     \setlength{\labelsep}{.5\parindent}
     \setlength{\leftmargin}{2\parindent}
     \setlength{\itemsep}{0pt}
     }%
   }%
  {\end{list}}
\newcommand{\Part}[1]{\item[\upshape#1]}

\renewcommand{\a}{\alpha}
\renewcommand{\b}{\beta}
\newcommand{\g}{\gamma}
\renewcommand{\d}{\delta}
\newcommand{\e}{\epsilon}
\newcommand{\f}{\phi}
\renewcommand{\l}{\lambda}
\renewcommand{\k}{\kappa}
\newcommand{\lhat}{\hat\lambda}
\newcommand{\m}{\mu}
\renewcommand{\o}{\omega}
\renewcommand{\r}{\rho}
\newcommand{\rbar}{{\bar\rho}}
\newcommand{\s}{\sigma}
\newcommand{\sbar}{{\bar\sigma}}
\renewcommand{\t}{\tau}
\newcommand{\z}{\zeta}

\newcommand{\D}{\Delta}

\newcommand{\gp}{{\mathfrak{p}}}
\newcommand{\gP}{{\mathfrak{P}}}
\newcommand{\gq}{{\mathfrak{q}}}
\newcommand{\gf}{{\mathfrak{f}}}

\newcommand{\Acal}{{\mathcal A}}
\newcommand{\Bcal}{{\mathcal B}}
\newcommand{\Ccal}{{\mathcal C}}
\newcommand{\Dcal}{{\mathcal D}}
\newcommand{\Ecal}{{\mathcal E}}
\newcommand{\F}{{\mathcal F}}
\newcommand{\Fcal}{{\mathcal F}}
\newcommand{\cF}{{\mathcal F}}
\newcommand{\Gcal}{{\mathcal G}}
\newcommand{\Hcal}{{\mathcal H}}
\newcommand{\Ical}{{\mathcal I}}
\newcommand{\Jcal}{{\mathcal J}}
\newcommand{\Kcal}{{\mathcal K}}
\newcommand{\Lcal}{{\mathcal L}}
\newcommand{\Mcal}{{\mathcal M}}
\newcommand{\Ncal}{{\mathcal N}}
\newcommand{\Ocal}{{\mathcal O}}
\newcommand{\Pcal}{{\mathcal P}}
\newcommand{\Qcal}{{\mathcal Q}}
\newcommand{\Rcal}{{\mathcal R}}
\newcommand{\Scal}{{\mathcal S}}
\newcommand{\Tcal}{{\mathcal T}}
\newcommand{\Ucal}{{\mathcal U}}
\newcommand{\Vcal}{{\mathcal V}}
\newcommand{\Wcal}{{\mathcal W}}
\newcommand{\Xcal}{{\mathcal X}}
\newcommand{\Ycal}{{\mathcal Y}}
\newcommand{\Zcal}{{\mathcal Z}}

\newcommand{\OO}{{\mathcal O}}    
\newcommand{\KK}{{\mathcal K}}    
\renewcommand{\O}{{\mathcal O}}   

\renewcommand{\AA}{\mathbb{A}}
\newcommand{\BB}{\mathbb{B}}
\newcommand{\CC}{\mathbb{C}}
\newcommand{\FF}{\mathbb{F}}
\newcommand{\GG}{\mathbb{G}}
\newcommand{\PP}{\mathbb{P}}
\newcommand{\NN}{\mathbb{N}}
\newcommand{\QQ}{\mathbb{Q}}
\newcommand{\RR}{\mathbb{R}}
\newcommand{\ZZ}{\mathbb{Z}}

\newcommand{\bfa}{{\mathbf a}}
\newcommand{\bfb}{{\mathbf b}}
\newcommand{\bfc}{{\mathbf c}}
\newcommand{\bfe}{{\mathbf e}}
\newcommand{\bff}{{\mathbf f}}
\newcommand{\bfg}{{\mathbf g}}
\newcommand{\bfp}{{\mathbf p}}
\newcommand{\bfr}{{\mathbf r}}
\newcommand{\bfs}{{\mathbf s}}
\newcommand{\bft}{{\mathbf t}}
\newcommand{\bfu}{{\mathbf u}}
\newcommand{\bfv}{{\mathbf v}}
\newcommand{\bfw}{{\mathbf w}}
\newcommand{\bfx}{{\mathbf x}}
\newcommand{\bfy}{{\mathbf y}}
\newcommand{\bfz}{{\mathbf z}}
\newcommand{\bfA}{{\mathbf A}}
\newcommand{\bfB}{{\mathbf B}}
\newcommand{\bfC}{{\mathbf C}}
\newcommand{\bfF}{{\mathbf F}}
\newcommand{\bfG}{{\mathbf G}}
\newcommand{\bfI}{{\mathbf I}}
\newcommand{\bfM}{{\mathbf M}}
\newcommand{\bfzero}{{\boldsymbol{0}}}
\newcommand{\bfmu}{{\boldsymbol\mu}}

\newcommand{\Aut}{\operatorname{Aut}}
\newcommand{\Div}{\operatorname{Div}}
\newcommand{\Prin}{\operatorname{Prin}}
\newcommand{\End}{\operatorname{End}}
\newcommand{\Gal}{\operatorname{Gal}}
\newcommand{\Image}{\operatorname{Image}}
\newcommand{\ord}{\operatorname{ord}}
\newcommand{\Pic}{\operatorname{Pic}}
\newcommand{\rank}{\operatorname{rank}}
\renewcommand{\setminus}{\smallsetminus}
\newcommand{\val}{{\operatorname{val}}}
\newcommand{\la}{{\langle}}
\newcommand{\ra}{{\rangle}}
\newcommand{\into}{\hookrightarrow}     
\def\onto{\twoheadrightarrow}           
\def\isomap{{\buildrel \sim\over\longrightarrow}} 

\newcommand{\exedout}{%
  \rule{0.8\textwidth}{0.5\textwidth}%
}

\newcommand{\dotcup}{\ensuremath{\mathaccent\cdot\cup}}


\title[Torsor structures on spanning trees]{The Bernardi process and torsor structures on spanning trees}

\author{Matthew Baker}
\email{mbaker@math.gatech.edu}
\address{School of Mathematics,
          Georgia Institute of Technology}
          
\author{Yao Wang}
\email{yw10@princeton.edu}
\address{Department of Mathematics,
          Princeton University}
          
\begin{abstract}
Let $G$ be a {\em ribbon graph}, i.e., a connected finite graph $G$ together with a cyclic ordering of the edges around each vertex.  By adapting a construction due to Olivier Bernardi, we associate to any pair $(v,e)$ consisting of a vertex $v$ and an edge $e$ adjacent to $v$ a bijection $\beta_{(v,e)}$ between spanning trees of $G$ and elements of the set $\Pic^g(G)$ of degree $g$ divisor classes on $G$, where $g$ is the {\em genus} of $G$ in the sense of Baker-Norine.  We give a new proof that the map $\beta_{(v,e)}$ is bijective by explicitly constructing an inverse. Using the natural action of the Picard group $\Pic^0(G)$ on $\Pic^g(G)$, we show that
the Bernardi bijection $\beta_{(v,e)}$ gives rise to a simply transitive action $\beta_v$ of $\Pic^0(G)$ on the set of spanning trees which does not depend on the choice of $e$.
A {\em plane graph} has a natural ribbon structure (coming from the counterclockwise orientation of the plane), and in this case we show that $\beta_v$ is
independent of $v$ as well.  Thus for plane graphs, the set of spanning trees is naturally a torsor for the Picard group.  Conversely, we show that if $\beta_v$ is independent of $v$
then $G$ together with its ribbon structure is planar.  We also show that the natural action of $\Pic^0(G)$ on spanning trees of a plane graph is compatible with planar duality.

\medskip

These findings are formally quite similar to results of Holroyd et al. and Chan-Church-Grochow, who used {\em rotor-routing} to construct an action $r_v$ of $\Pic^0(G)$ on the spanning trees of a ribbon graph $G$, which they show is independent of $v$ if and only if $G$ is planar.  It is therefore natural to ask how the two constructions are related.  We prove
that $\beta_v = r_v$ for all vertices $v$ of $G$ when $G$ is a planar ribbon graph, i.e. the two torsor structures (Bernardi and rotor-routing) on the set of spanning trees coincide.  In particular, it follows that the rotor-routing torsor is compatible with planar duality.
We conjecture that for every non-planar ribbon graph $G$, there exists a vertex $v$ with $\beta_v \neq r_v$.
\end{abstract}

\thanks{We thank Spencer Backman, Melody Chan, and Dan Margalit for helpful discussions and feedback on an earlier draft.  We also thank 
Chi Ho Yuen for catching a sign error in Theorem~\ref{theorem:planar_duality}, and the anonymous referees for their extraordinarily
careful proofreading and numerous useful suggestions.  This work began 
at the American Insitute of Mathematics workshop ``Generalizations of chip firing and the critical group'' in July 2013, and we would like to thank AIM as well as the organizers of that conference (L. Levine, J. Martin, D. Perkinson, and J. Propp) for providing a stimulating environment.  The first author was supported in part by NSF grants DMS-1201473 and DMS-1529573.}

\maketitle


\section{Introduction}

If $G$ is a connected graph on $n$ vertices, the {\em Picard group} $\Pic^0(G)$ of $G$ (also called the sandpile group, critical group, or Jacobian group) is a finite abelian group whose cardinality is the determinant of any $(n-1)\times (n-1)$ principal sub-minor of the Laplacian matrix of $G$.  By Kirchhoff's Matrix-Tree Theorem, this quantity is equal to the number of spanning trees of $G$.
There are several known families of bijections between spanning trees and elements of $\Pic^0(G)$, which we think of as giving {\em bijective proofs} of Kirchhoff's Theorem -- see for example \cite{Baker-Shokrieh} and the references therein.
However, such bijections depend on various auxiliary choices, and there is no canonical bijection in general between $\Pic^0(G)$ and the set $S(G)$ of spanning trees of $G$ (see \cite[p. 2]{CCG}).

\medskip

Jordan Ellenberg asked if it might be the case that, under certain conditions, $S(G)$ is naturally a {\em torsor}\footnote{A {\em torsor} for a group $H$ is a set $S$ together with a simply transitive action of $H$ on $S$, i.e. an action such that for every $x,y \in S$ there is a unique $h \in H$ for which $h \cdot x = y$.  The existence of such an action implies, in particular,
that $|H|=|S|$.} for $\Pic^0(G)$.
This question was thoroughly studied in the paper \cite{CCG}, where it was established via the {\em rotor-routing process} that given a {\em plane graph}\footnote{In this paper we 
distinguish between a {\em planar graph}, which is a graph $G$ that can be embedded in $\RR^2$ with no crossings, and a {\em plane graph}, which is a graph $G$ together with such an embedding.} there is a canonical simply transitive action of $\Pic^0(G)$ on $S(G)$.
More generally, if one fixes a {\em ribbon structure} on $G$ and then chooses a root vertex $v$,
rotor-routing produces a simply transitive action $r_v$ of $\Pic^0(G)$ on $S(G)$ which is shown in \cite{CCG} to be independent of $v$ {\em if and only if} $G$ together with its ribbon structure is planar.

\medskip

The first author learned of the results of \cite{CCG} at a 2013 AIM workshop on ``Generalizations of Chip Firing'', and at the same workshop he learned about an interesting family of bijections due to Olivier Bernardi \cite{Bernardi} between spanning trees, root-connected out-degree sequences, and recurrent sandpile configurations.  Bernardi's bijections depend on choosing a ribbon structure, a root vertex $v$, and an edge $e$ adjacent to $v$.
It was natural to ask whether Bernardi's bijections became torsor structures upon forgetting the edge $e$, and whether the resulting torsor was again independent of $v$ in the planar case.  The present paper answers these questions affirmatively.  Specifically, given a ribbon graph $G$, we show that:

\begin{enumerate}
\item[(1)] (Theorem~\ref{theorem:independent_of_e}) Fixing a vertex $v$, the Bernardi process defines a simply transitive action $\beta_v$ of $\Pic^0(G)$ on $S(G)$.
\item[(2)] (Theorems~\ref{theorem:planar_canonical} and \ref{theorem:non-planar_non-canonical}) The action $\beta_v$ is independent of $v$ if and only if the ribbon graph $G$ is planar.
\item[(3)] (Theorem~\ref{theorem:bernardi_rotor-routing_comparison}) If $G$ is planar, the Bernardi and rotor-routing torsors coincide.
\end{enumerate}

We also give an example which shows that if $G$ is non-planar then $\beta_v$ and $r_v$ do not in general coincide.
In fact, we conjecture that if $G$ is non-planar then there {\em always} exists a vertex $v$ such that $\beta_v \neq r_v$.

\medskip

We also investigate the relationship between the Bernardi process and planar duality.  It is known that if $G$ is a planar graph then there is a canonical isomorphism
$\Psi : \Pic^0(G) \isomap \Pic^0(G^\star)$ and a canonical bijection $\sigma: S(G) \isomap S(G^\star)$.
It is thus natural to ask whether the Bernardi torsor is compatible with these isomorphisms.  We prove that this is indeed the case:

\begin{enumerate}
\item[(4)] (Theorem~\ref{theorem:planar_duality}) If $G$ is planar, the following diagram is commutative:
\[
\begin{CD}
\Pic^0(G) \times S(G) @>{\beta}>> S(G) \\
@VV{\Psi \times \sigma}V	@VV{\sigma}V \\
\Pic^0(G^\star) \times S(G^\star) @>{\beta^\star}>> S(G^\star) \\
\end{CD}
\]
\end{enumerate}

Combined with (3), this proves that the rotor-routing torsor is also compatible with planar duality (which the first author had conjectured at the 2013 AIM workshop).  An independent proof (not going through the Bernardi process) of the compatibility of the rotor-routing torsor with planar duality has recently been given by \cite{Chan-et-al}.

\medskip

A key technical insight which we use repeatedly is an interpretation of Bernardi's bijections in terms of {\em (integral) break divisors} in the sense
of \cite{ABKS}.  Break divisors are in canonical bijection with elements of $\Pic^g(G)$, where $g$ is the {\em genus} of $G$ in the sense of
\cite{BN1}, and $\Pic^g(G)$ is canonically a torsor for $\Pic^0(G)$.  Among other things, we use break divisors to give a different proof from the
one in \cite{Bernardi} that Bernardi's maps are in fact bijections.
In particular, we are able to give a new family of {\em bijective proofs} of the Matrix-Tree Theorem.

\medskip

The Bernardi and rotor-routing processes are defined quite differently, so it seems rather miraculous that the two stories parallel one another so closely, intertwining in the planar case and diverging in general.  It is a strangely beautiful tale which still appears to hold some mysteries.

\medskip

The plan of this paper is as follows.  In Section~\ref{sec:background} we review the necessary background material, including Picard groups, break divisors, ribbon graphs, rotor-routing, and the Bernardi process.  In Section~\ref{sec:bijective} we give our new proof that Bernardi's maps are bijections, and in Section~\ref{sec:Bernardi_torsor} we define the action $\beta_v$ and establish (1).  In Section~\ref{sec:planar_independence} we investigate planarity and establish (2), and in Section~\ref{sec:planar_duality} we study the relationship between planar duality and the Bernardi process and establish (4).  Finally, in Section~\ref{sec:comparison}, we relate the Bernardi and rotor-routing torsors in the planar case and establish (3).

\section{Background}
\label{sec:background}

\subsection{Graphs, divisors, and linear equivalence}

Let $G$ be a {\em graph}, by which we mean a finite connected multigraph, possibly with loop edges.
We denote the vertex set by $V(G)$ and the edge set by $E(G)$, and we let $n$ be the number of vertices.
We denote by $\vec{e}$ an edge $e \in E(G)$ together with an orientation, and by $(\vec{e})^{\rm op}$ the edge $e$ with the opposite
orientation.

\medskip

Following \cite{BN1}, we define the group of {\em divisors} on $G$, denoted $\Div(G)$, to be the free abelian group on $V(G)$.
We write a divisor $D$ as $\sum_{v \in V(G)} a_v (v)$, where the $a_v$ are integers and $(v)$ is a formal symbol for the generator of $\Div(G)$ corresponding to $v$.
The {\em degree} of $D$ is defined to be $\sum a_v$, and the set of divisors of degree $d$ is denoted $\Div^d(G)$.

\medskip

Let $M(G)$ be the set of functions $f : V(G) \to \ZZ$.
We define the group of {\em principal divisors} on $G$, denoted $\Prin(G)$, to be
$ \{ \Delta f \; : \; f \in M(G) \}$, where
\[
\Delta f = \sum_{v \in V(G)} \left( \sum_{e = vw} \left( f(v) - f(w) \right) \right) (v)
\]
is the Laplacian of $f$ considered as a divisor on $G$.
We say that two divisors $D$ and $D'$ are {\em linearly equivalent}, written $D \sim D'$, if $D-D' \in \Prin(G)$.

\medskip

For each $d \in \ZZ$ we let $\Pic^d(G)$ be the set of linear equivalence classes of degree $d$ divisors on $G$.
In particular, $\Pic^0(G) = \Div^0(G) / \Prin(G)$ is a {\em group} which acts simply and transitively by addition on each $\Pic^d(G)$.
By basic linear algebra, the cardinality of $\Pic^0(G)$ (and hence of every $\Pic^d(G)$) is the determinant of any $(n-1)\times (n-1)$ principal sub-minor of the Laplacian matrix of $G$, 
where $n = |V(G)|$.

\medskip

We denote by $[D] \in \Pic^d(G)$ the linear equivalence class of a divisor $D \in \Div^d(G)$.

\subsection{Break divisors}

Let $G$ be a graph, and let $g = g_{\rm comb}(G)$ be the {\em combinatorial genus}\footnote{Note that graph theorists often use the term {\em genus} to denote the minimal topological genus of a ribbon structure on $G$ in the sense of \S\ref{sec:ribbon_graphs} below.  However, to highlight the connections with algebraic geometry in
the spirit of \cite{BN1} we will use the unmodified term {\em genus} to denote the combinatorial genus of $G$.} of $G$, defined as $\# E(G) - \# V(G) + 1$.
If $T$ is a spanning tree of $G$, then there are exactly $g$ edges $e_1,\ldots,e_g$ of $G$ not belonging to $T$.
A divisor of the form $D = \sum_{i=1}^g (v_i)$, where $v_i$ is an endpoint of $e_i$, is called a {\em $T$-break divisor}.
(In this situation we say that $D$ is {\em compatible with $T$}, or that $T$ is {\em compatible with $D$}.) 
A {\em break divisor}\footnote{In \cite{ABKS}, these are called {\em integral break divisors}.  We omit the adjective `integral' because non-integral break divisors play no role in this paper.} is a $T$-break divisor for some spanning tree $T$.
We denote by $B(G)$ the set of break divisors on $G$.

\medskip

The following important fact is proved in \cite[Theorem 4.25]{ABKS}; 
it implies, in particular, the surprising fact that the number of break divisors on $G$ is equal to the number of spanning trees:
\begin{theorem} \label{theorem:ABKS4.25}
Every element of $\Pic^g(G)$ is linearly equivalent to a unique break divisor.
\end{theorem}

The following is a simple but useful result:

\begin{lemma}
\label{lemma:eulerchar}
The restriction of a break divisor on $G$ to a connected induced subgraph $H$ has degree at least the genus $g(H)$ of $H$.
\end{lemma}

\begin{proof}
Let $D$ be a break divisor on $G$, and let $T$ be a spanning tree compatible with $D$.  The restriction of $T$ to $H$ is a spanning forest $F$ of 
$H$ and hence $|E(F)| \leq |V(H)| - 1$.  By definition, if $e_1,\ldots,e_g$ are the edges of $G$ not belonging to $T$, we can write $D = \sum_{i=1}^g (v_i)$ 
where $v_i$ is an endpoint of $e_i$.  Let $D' = \sum_{e_i \in E(H) \backslash E(F)} (v_i)$.  Then $D|_H \geq D'$ and thus
\[
{\rm deg}(D|_H) \geq {\rm deg}(D') = |E(H)| - |E(F)| \geq |E(H)| - |V(H)| + 1 = g(H).
\]
\end{proof}

Although we will not need it in this paper, Lemma~3.3 and Proposition~4.8 of \cite{ABKS} imply, conversely, that if $D$ is a divisor on $G$ of degree 
$g(G)$ and the restriction of $D$ to every connected induced subgraph $H$ has degree at least $g(H)$, then $D$ is a break divisor.

\subsection{Ribbon graphs} \label{sec:ribbon_graphs}

A {\em ribbon graph} is a finite graph $G$ together with a cyclic ordering of the edges around each vertex.
A ribbon structure on $G$ gives an embedding of $G$ into a canonical (up to homeomorphism) closed orientable surface $S$.
(The surface $S$ is obtained by first thickening $G$ to a compact orientable surface-with-boundary $R$ and then gluing a disk to each boundary component of $R$.)
Conversely, every such embedding gives rise to a ribbon structure on $G$.  Ribbon structures are therefore sometimes called {\em combinatorial embeddings}.\footnote{Another name for ribbon structures, used widely in the topological graph theory literature, is {\em rotation systems}.}  (A good reference for basic combinatorial properties of ribbon graphs is \cite{Thomassen}.)
We refer to the genus of $S$ as the {\em topological genus} $g_{\rm top}(G)$ of $G$, and say that $G$ is {\em planar} if $g_{\rm top}(G)=0$.

\medskip

Equivalently, a planar ribbon graph is one which can be embedded in the Euclidean plane $\RR^2$ without crossings in such a way that the
ribbon structure on $G$ is induced by the natural counterclockwise orientation\footnote{There are two ways to orient $\RR^2$, and for most of this paper we will implicitly or explicitly work with the counterclockwise orientation.  However, for planar duality it is important to consider both orientations.} on $\RR^2$.

\medskip

Every closed orientable surface $S$ can be cut along a collection of loops to give a polygon $P$, and conversely by identifying certain pairs of sides of $P$ one can recover the surface $S$.
In fact, as is well-known (see e.g. \cite[p.~126]{Henle}), one can arrange for the labeling of the edges of $P$ as one traverses the perimeter counterclockwise to have the form
\begin{equation} \label{eq:fund_poly}
P = a_1 b_1 a_1^{-1} b_1^{-1} \cdots a_g b_g a_g^{-1} b_g^{-1}.
\end{equation}
(In this labeling, $a_i$ and $a_i^{-1}$ get glued together with opposite orientations, and similarly for $b_i$ and $b_i^{-1}$.)
We define the {\em genus} of $P$ to be the integer $g$, which is also the genus of $S$.

\medskip

In particular (by avoiding the vertices), every ribbon graph $G$ can be drawn inside a {\em fundamental polygon} $P$ whose boundary is glued as in (\ref{eq:fund_poly}), with all vertices of $G$ lying on the interior of $P$ (see Figure~\ref{fig:T1-T2} for an example).  One may take the genus of $P$ to be the topological genus of $G$.

\subsection{Planar duality}
\label{sec:planar_duality_background}

It is well-known that every plane graph has a {\em planar dual} $G^\star$ whose vertices correspond to faces of $G^\star$ and whose edges are dual to edges of $G$.
Duality for plane graphs has the following well-known properties:

\begin{enumerate}
\item There is a canonical isomorphism $G^{\star\star} \cong G$.
\item There is a canonical bijection $\psi : S(G) \to S(G^\star)$ sending a spanning tree $T$ of $G$ to the spanning tree $T^*$ whose edges are dual to the edges of $G$ {\em not} in $T$.
\item There is a isomorphism of groups $\Psi = \Psi_{\mathcal O}: \Pic^0(G) \isomap \Pic^0(G^\star)$ depending on the choice of an orientation ${\mathcal O}$ of the plane.
\end{enumerate}

The isomorphism in (3) is obtained as follows.
The choice of ${\mathcal O}$ allows us to identify directed edges of $G$ with directed edges of $G^\star$ in a natural way: if $\vec{e}$ is a directed
edge of $G$ then locally, near the crossing of $e$ and $e^\star$, one obtains an orientation on $e^\star$ by rotating in the direction 
{\em opposite} from ${\mathcal O}$; this convention is needed in order to make the diagram in Theorem~\ref{theorem:planar_duality} commute.
(So if, for example, the plane is oriented counterclockwise for $G$, then one gets from $\vec{e}$ to $(\vec{e})^\star$ by a clockwise rotation.)
This identification affords an isomorphism $\psi$ from the lattice $C_I$ of integral $1$-chains on $G$ to the lattice $C_I^\star$ of integral $1$-chains on $G^\star$.
By \cite[Proposition 8]{BdlHN}, there is also a canonical isomorphism between the {\em lattice of integer flows} $Z_I$ for $G$ and the
{\em lattice of integer cuts} $B^\star_I$ for $G^\star$, and vice-versa.
And by \cite[Proposition 28.2]{Biggs97}, there is a canonical isomorphism $\Pic^0(G)\isomap \frac{C_I}{Z_I \oplus B_I}$ whose inverse is induced by the boundary map $\partial : C_I \to \Div^0(G)$.
We thus obtain an isomorphism $\Psi  : \Pic^0(G) \isomap \Pic^0(G^\star)$ induced by the composition
\[
\Pic^0(G)\isomap \frac{C_I}{Z_I \oplus B_I} \isomap \frac{C_I^\star}{Z^\star_I \oplus B^\star_I} \isomap \Pic^0(G^\star).
\]

\medskip

If $G$ is a planar ribbon graph (with respect to some orientation ${\mathcal O}$ of the plane), the natural way to define a dual planar ribbon graph $G^\star$ is to use the {\em opposite} orientation ${\mathcal O}^{\rm op}$ to define the cyclic ordering around each vertex of the dual graph.  With this convention, facts (1)-(3) above also hold for planar ribbon graphs (with the map in (3) now being canonical).

\subsection{Rotor-routing}
\label{sec:rotor-routing}
We give a quick summary of some basic facts about rotor-routing from \cite{Holroyd-et-al} and \cite{CCG} which will be needed for our proof of
Theorem~\ref{theorem:bernardi_rotor-routing_comparison}.

\medskip

Let $G$ be a ribbon graph, and choose a {\em sink vertex} $y$ of $G$.  The {\em rotor-routing model} is a deterministic process on {\em states} $(\rho,x)$, where $\rho$ is a {\em rotor configuration} (i.e., an assignment of an outgoing edge $\rho[z]$ to each vertex $z \neq y$ of $G$) and $x$ is a vertex of $G$, which one thinks of as the position of a {\em chip} which moves along the graph.  Each step of the rotor-routing process consists of replacing $(\rho,x)$ with a new state $(\rho',x')$, where $\rho'$ is obtained from $\rho$ by rotating the rotor $\rho[x]$ to the next edge $\widetilde{{\rho}[x]}$ in the cyclic order at $x$ and $x'$ is the other endpoint of $\widetilde{{\rho}[x]}$.
We think of the chip as moving from $x$ to $x'$ along the edge $\widetilde{{\rho}[x]}$ in the process.

\medskip

Given a root vertex $y$, a vertex $x$, and a spanning tree $T$, one defines a new spanning tree $\left( (x)-(y) \right)_y(T)$ as follows.
Orienting the edges of $T$ towards $y$ gives a rotor configuration $\rho_T$ on $G$.
Place a chip at the initial vertex $x$ and iterate the rotor-routing process starting with the pair $(\rho,x)$ until the chip first reaches $y$ (which it always does, see \cite[Lemma 3.6]{Holroyd-et-al}).  Call the resulting pair $(\sigma,y)$.  Denote the pairs at each step of the rotor-routing process
by $(\rho_0,x_0)=(\rho,x),(\rho_1,x_1),\ldots,(\rho_k,x_k)=(\sigma,y)$.
Define $S_0 = T$, and for $i=0,\ldots,k-1$ define a subset $S_{i+1} \subset E(G)$ by $S_{i+1} = S_i \backslash \rho_i[x_i] \cup \widetilde{{\rho_i}[x_i]}$.  Although it is not in general true that each $S_{i+1}$ is a spanning tree\footnote{In general, $S_{i+1}$ will either be a spanning tree or the 
union of a unicycle $C'$ and a tree $T'$ such that $C' \cup T'$ contains every vertex of $G$.}, it is proved in \cite{Holroyd-et-al} that $S_k$ is a spanning tree $T'$,
and we define $\left( (x)-(y) \right)_y(T) = T'$. (See Figure~\ref{fig:Rotor_routing} for an example.)
It is proved in \cite{Holroyd-et-al} that this action extends by linearity to an action of $\Div^0(G)$ on $S(G)$, which by {\em loc.~cit.~} is trivial on $\Prin(G)$ and descends to a simply transitive action $r_y$ of $\Pic^0(G)$ on $S(G)$.

\begin{figure}
\centering
\includegraphics[width=.7\textwidth]{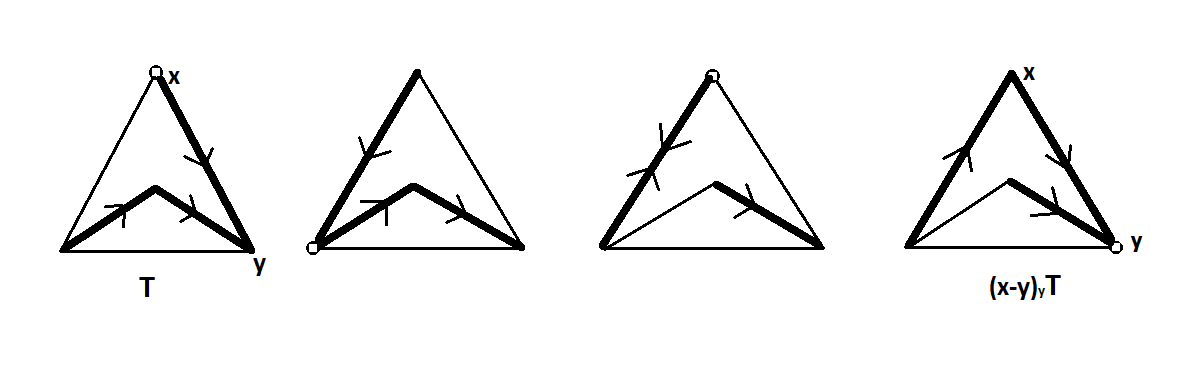}
\caption{An example of the rotor-routing process.}
\label{fig:Rotor_routing}
\end{figure}

\medskip

The following result is proved in \cite{CCG}:

\begin{theorem} \label{thm: planar_rotor-routing}
The action $r_y$ is independent of the root vertex $y$ if and only if the ribbon graph $G$ is planar.
\end{theorem}

An important ingredient in the proof of Theorem~\ref{thm: planar_rotor-routing} is the relationship between rotor-routing and unicycles.  A {\em unicycle} is a state $(\rho,x)$ such that $\rho$ contains exactly one directed cycle $C(\rho)$ and $x$ lies on this cycle.  

\medskip

Suppose $G$ has $m$ edges and $(\rho,x)$ is a unicycle on $G$.  By \cite[Lemma 4.9]{Holroyd-et-al}\footnote{The authors of \cite{Holroyd-et-al} define rotor-routing slightly differently than we do here: instead of using a configuration with a sink, they use ``sink-free rotor routing''.  One can easily translate back and forth between our process and theirs by adding an outgoing edge to the sink; this makes the configuration a unicycle instead of a spanning tree. During any stage of the process, sink-free rotor-routing takes unicycles to unicycles \cite[Lemma 3.3]{Holroyd-et-al}.}, if we iterate the rotor-routing process $2m$ times starting at the state $(\rho,x)$, the chip traverses each edge of $G$ exactly once in each direction, each rotor makes exactly one full turn, and the stopping state is $(\rho,x)$.
Using this, one sees that the relation on unicycles defined by $(\rho,x) \sim (\rho',x')$ iff $(\rho',x')$ can be obtained from $(\rho,x)$ by iterating the rotor-routing process some number of times is an equivalence relation.

\medskip

Given a unicycle $(\rho,x)$, denote by $\bar{\rho}$ the configuration obtained from $\rho$ by reversing the edges of $C$ and keeping all other rotors the same.  One says that the unicycle $(\rho,x)$ is {\em reversible} if $(\rho,x) \sim (\bar{\rho},x)$.  By \cite[Proposition 7]{CCG}, the notion of reversibility is intrinsic to the directed cycle $C(\rho)$, and does not actually depend on $\rho$ or $x$; in other words, if $(x,\rho)$ and $(x',\rho')$ are unicycles with $C(\rho)=C(\rho')$, then $(x,\rho)$ is reversible iff $(x',\rho')$ is.  It therefore makes sense to talk about reversibility of directed cycles in a ribbon graph $G$.  The importance of this concept stems from \cite[Proposition 9]{CCG}, which asserts that the ribbon graph $G$ is planar if and only if every directed cycle of $G$ is reversible.

\subsection{The Bernardi process}
\label{sec:BernardiProcess}

Let $G$ be a ribbon graph, and fix a pair $(v,e)$ (which we refer to as the {\em initial data}) consisting of a vertex $v$ and an edge $e$ adjacent to $v$.
In this section, we recall the {\em tour} of $G$ which Bernardi \cite[\S{3.1}]{Bernardi} associates to the initial data $(v,e)$ together with a spanning tree $T$, and describe how to associate a break divisor to this tour.

\medskip

Let $T$ be a spanning tree of $G$.  The tour $\tau_{(v,e)}(T)$ is a traversal of
$T$ which begins and ends at $v$.  Informally, the tour is obtained by walking
along edges belonging to $T$ and cutting through edges not belonging to $T$,
beginning with $e$ and proceeding according to the ribbon structure.\footnote{From a topological point of view, the tour $\tau_{(v,e)}(T)$ is obtained by traversing the boundary of a small $\epsilon$-neighborhood of $T$ in the oriented surface $S$ on which the ribbon graph is embedded.}  (See Figure~\ref{fig:bernardi_tour}.)

\begin{figure}
\centering
\includegraphics[width=.4\textwidth]{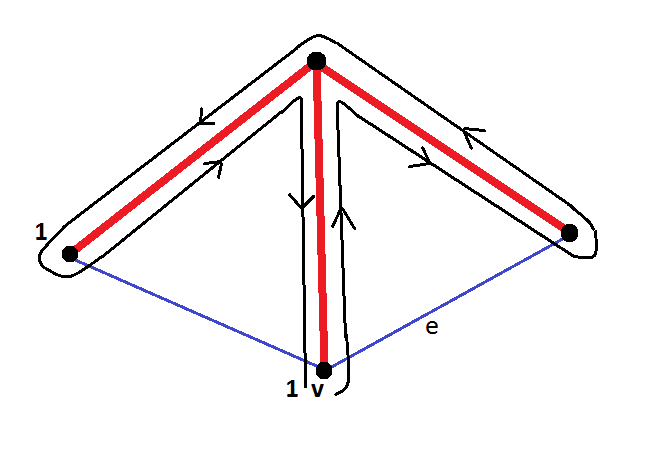}
\includegraphics[width=.4\textwidth]{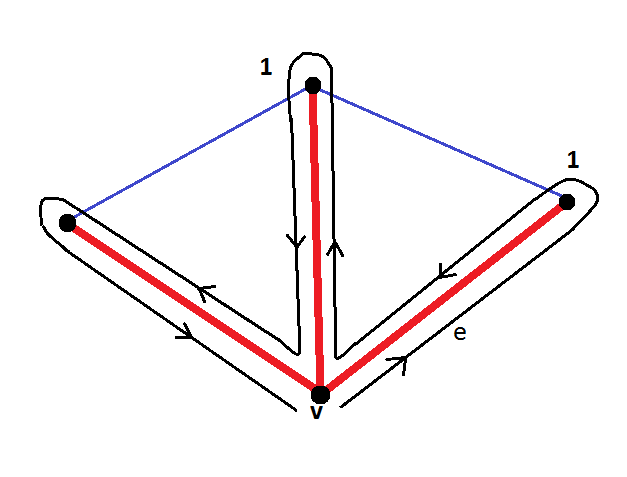}
\caption{The Bernardi tour associated to two different spanning trees (shown in red),
  along with the corresponding break divisors.  The ribbon structure on $G$ is
  induced by the counterclockwise orientation of the plane.}
  \label{fig:bernardi_tour}
\end{figure}

\medskip

More formally, the tour is a sequence
\[
\tau_{(v,e)}(T) = (v_0,\vec{e}_1,v_1,\vec{e}_2,\ldots,\vec{e}_k,v_k)
\]
where each $v_i$ is a vertex of $G$ and $\vec{e}_i$ is a directed edge of $G$ leading to $v_i$.
We set $v_0 = v$.  If $e \in T$, we define $e_1$ to be $e$, and if $e \not\in T$, we let $e_1$ be the first edge after $e$ in the cyclic ordering around $v$.
Let $v_1$ be the other endpoint (besides $v_0$) of $e_1$.  
The $(v_i,e_i)$ for $2 \leq i \leq k$ are defined inductively by declaring that
$e_i = (v_{i-1},v_i)$ is the first edge after $e_{i-1}$ belonging to $T$ in the cyclic ordering of the edges around $v_{i-1}$.
The tour stops when each edge of $T$ has been included twice among the $e_i$ with $1 \leq i \leq k$, once with each orientation; necessarily we will have $v_k = v$.
Note that different choices of initial data give rise to tours which are cyclic shifts of one another.

\medskip

The break divisor $\beta_{(v,e)}(T)$ associated to the tour is obtained by dropping a chip at the corresponding vertex each time the tour first cuts through an edge not belonging to $T$.
In other words, for each edge $e'$ not in the spanning tree, let $\{ \vec{e}_i, \vec{e}_j \}$ be the two oriented edges whose underlying unoriented edge is $e'$, with $i < j$, and set $\eta_{v,e}(e') := v_{i-1}$.  We define
\[
\beta_{(v,e)}(T) := \sum_{e' \not\in T} (\eta_{v,e}(e')).
\]

\medskip

The amazing fact implicitly discovered by Bernardi is that the association $T \mapsto \beta_{(v,e)}(T)$ gives a {\em bijection} between spanning trees of $G$ and break divisors.\footnote{Bernardi phrases his result (Theorem 41(5)) in terms of out-degree sequences of orientations, but in view of the results of \cite{ABKS} the two points of view are equivalent.}
In the next section, we give a proof which is different from Bernardi's that $\beta_{(v,e)}$ is bijective.
In addition to making the present paper more self-contained, our proof of Bernardi's theorem involves a new recursive procedure which might be of independent interest.
However, the reader already familiar with Bernardi's proof of Theorem 41(5) (which in particular makes use of his Propositions 18 and 34) may safely skip the next section
if desired.

\section{The Bernardi map is bijective}
\label{sec:bijective}

Let $S(G)$ denote the set of spanning trees of a graph $G$, and let $B(G)$ denote the set of break divisors of $G$, which is canonically isomorphic to $\Pic^g(G)$.
As in the previous section, we fix a pair $(v,e)$ consisting of a vertex $v$ and an edge $e$ adjacent to $v$.
In this section we give a new proof of the fact that the map $\beta := \beta_{(v,e)} : S(G) \to B(G)$ is bijective by explicitly constructing an inverse map 
$\alpha$.

\subsection{Definition of the inverse map}

Let $D$ be a break divisor and fix a pair $(v,e)$ consisting of a vertex $v$ and an edge $e$ adjacent to $v$.
To define a map $\alpha: B(G) \to S(G)$ which is inverse to $\beta$, we will inductively construct a spanning tree $T$ with $\beta(T)=D$, along with a corresponding Bernardi tour which traverses $T$.  Since the tour is obtained by walking along edges belonging to $T$ and cutting through edges not belonging to $T$, but in this case we don't know $T$, our challenge is to use the break divisor $D$ to figure out which edges to walk along and which to cut through.
The solution is that we will cut through an edge $e'$ if removing that edge from the graph and subtracting a chip from $D$ at the current vertex gives a break divisor $D'$ on the resulting (connected) graph $G'$; otherwise we walk along $e'$ and add it to the spanning tree which we're building.

More formally, $\alpha(D)$ is defined to be the output of Algorithm~\ref{RightBernardiInverse} below.

\begin{algorithm}
\caption{Inverse to the Bernardi map}
\KwIn{A connected graph $G$ and a break divisor $D\in \Div^g(G)$.}
\KwOut{A spanning tree $T$.}
\BlankLine
Set $i:=0$, $T := \emptyset$,$v' := v$, and $e' := e$.
\BlankLine
\While{$T \neq G$}
{Let $w'$ be the other endpoint (besides $v'$) of $e'$. \\
Let $G'$ be the graph obtained from $G$ by deleting $e'$. \\ Define $D' := D - (v') \in \Div^{g-1}(G')$. \\
\eIf{$G'$ is connected and $e' \not\in T$ and $D'$ is a break divisor on $G'$}{Replace $G$ by $G'$, $D$ by $D'$, and $e'$ by the edge following it in the (induced) cyclic ordering around $v'$ on $G'$.}{Replace $T$ by $T \cup \{ e' \}$, $v'$ by $w'$, and $e'$ by the edge following it in the cyclic ordering around $w'$ on $G'$.}
}
\BlankLine
Output $T$.
\label{RightBernardiInverse}
\end{algorithm}

\begin{proposition}
\label{prop:rightinverse}
Let $G$ be a graph and $D$ a break divisor on $G$.  Then $\alpha(D)$ is a spanning tree of $G$ and $\beta(\alpha(D)) = D$.
\end{proposition}

\begin{proof}
The proof is by induction on $|E(G)|$.  For the inductive step, given a ribbon graph $G$ and an edge $e=vv'$, we will need to endow the deletion $G\backslash e$ and contraction $G/e$ with the structure of ribbon graphs.  
For $G \backslash e$, we just remove $e$ from the cyclic ordering around $v$ and $v'$.  
In $G/e$, $v$ and $v'$ collapse to a single vertex $v''$, and if the cyclic ordering around $v$ is $e_1=e,\ldots,e_s$ and around $v'$ is $e_1'=e,\ldots,e_t'$ then the cyclic ordering around $v''$ is $e_2,\ldots,e_s,e_2',\ldots,e_t'$.

Starting with a break divisor $D$ on $G$ and initial data $(v,e)$, let $v'$ be the other endpoint of $e$.
If $D-(v)$ is a break divisor on $G \backslash e$, set $G'' := G\backslash e$ and $D'' := D-(v)$.  
Otherwise set $G'' := G/e$ and $D''(w) := D(w)$ for $w \neq v,v'$, $D''(v''): = D(v)+D(v')$.  

We claim that in either case, $D''$ is a break divisor on $G''$.
This is clear if $G'' = G \backslash e$, so we may assume that $G'' = G/e$.  
Let $T$ be a spanning tree of $G$ compatible with $D$.  If $e \in T$, then $T/e$ is a spanning tree of $G/e$ compatible with $D''$.  
If $e \not\in T$, then because $D-(v)$ is not a break divisor on $G \backslash e$, $e$ sends its chip to $v'$ rather than $v$.  Because of this, if $e'$ is an edge of $T$ incident to $v'$, $T' := \left( T \backslash \{ e \}\right)  \cup \{ e' \}$ is a spanning tree of $G$ compatible with $D$, and $T'/e$ is a spanning tree of $T/e$ compatible with $D''$. This proves the claim.

We now proceed with the inductive argument.  
There are two base cases to check, when $G$ is a loop and when $G$ is an edge, both of which are trivial.
By induction on $|E(G)|$ and the claim, $\alpha(D'')$ is a spanning tree $T''$ of $G''$, which corresponds naturally to a spanning tree $T$ of $G$.  
(If $G'' = G \backslash e$, set $T = T''$, and if $G'' = G/e$ set $T=T'' \cup \{ e \}$.)
One checks easily from the definitions that $\alpha(D) = T$ and $\beta(T) = D$.  Thus $\alpha$ is well-defined and $\beta \circ \alpha$ is the identity map on the set of break divisors. 
\end{proof}

\begin{remark}
There is an efficient (polynomial-time) algorithm for deciding whether or not a given divisor on a graph is a break divisor; see \cite{Backman}.
\end{remark}

\begin{corollary}
\label{cor:BernardiSurjective}
The Bernardi map $\beta : S(G) \to B(G)$ is surjective.
\end{corollary}

By Theorem~\ref{theorem:ABKS4.25} (Theorem~4.25 from \cite{ABKS}), we know that $|B(G)| = |S(G)|$, and thus Corollary~\ref{cor:BernardiSurjective}
implies that $\beta$ is an isomorphism and $\alpha$ is its inverse.  In the proof of the following result,
we argue directly that $\alpha$ is a left inverse to $\beta$ without using Theorem~\ref{theorem:ABKS4.25}.  

\begin{proposition}
\label{prop:leftinverse}
If $T$ is a spanning tree of $G$ then $\alpha(\beta(T)) = T$.
\end{proposition}

\begin{proof}
Let $D = \beta(T)$.  If $e \not\in T$ then $D - (v)$ is a break divisor on $G \backslash e$, and the result follows by induction on $|E(G)|$ as in the proof of Proposition~\ref{prop:rightinverse} (applying the inductive hypothesis to $G \backslash e$).  
If $D - (v)$ is not a break divisor on $G \backslash e$, then by the claim in the proof of Proposition~\ref{prop:rightinverse} 
the divisor $D''$ defined there is a break divisor on $G / e$.  If $e \in T$, we may then apply induction to $G / e$ and the result again follows.

Therefore it suffices to prove that the potentially troublesome case where $e \in T$ and $D - (v)$ is a break divisor on $G \backslash e$ does not actually occur.
Suppose for the sake of contradiction that it does.  Let $A$ be the connected component of $v$ in $T \backslash \{ e \}$, and let $G' = G[A]$ be
the corresponding induced subgraph of $G$.  Let $g'$ be the genus of $G'$, let $T'$ be the spanning tree of $G'$ given by the restriction of $T$, and let
$D'$ be the restriction of $D$ to $G'$.

Let $e'$ be the first edge in the cyclic ordering around $v$, starting with $e$, which belongs to $G'$, and let $\beta'$ be the Bernardi process on $G'$ with initial data $(v,e')$.  
Note that the Bernardi process $\beta$ on $G$ first walks along $e$, then tours along the vertices in the complement of $A$, then walks along $e$ again, and finishes with a tour of the vertices in $A$.  In addition, every edge in $G \backslash T$ which connects $A$ to its complement is crossed
before the tour of the vertices in $A$ begins.
It follows that $\beta'(T') = D'$.

In particular, $D'$ is a break divisor on $G'$ so the degree of $D'$ is equal to $g'$.
Therefore the restriction of $D-(v)$ to $G'$ has degree $g' - 1$.  
However, since $G'$ is a connected subgraph of $G \backslash e$, Lemma~\ref{lemma:eulerchar}
contradicts the assumption that $D-(v)$ is a break divisor on $G \backslash e$.
\end{proof}

\begin{corollary} \label{cor:BernardiBijective}
The Bernardi map $\beta : S(G) \to B(G)$ is bijective.
\end{corollary}

\begin{remark}
The proofs of Propositions~\ref{prop:rightinverse} and \ref{prop:leftinverse} show that the map $\gamma : B(G) \to B(G \backslash e) \dotcup B(G / e)$ sending $D$ to $D-(v) \in B(G \backslash e)$ if $D-(v)$ is a break divisor on $G \backslash e$ and to $D'' \in B(G\backslash e)$ otherwise, where $D''(w) := D(w)$ for $w \neq v,v'$ and $D''(v''): = D(v)+D(v')$, is bijective.  In particular, $|B(G)| = |B(G \backslash e)| + |B(G / e)|$ for every $e \in E(G)$.
Since $|S(G)|$ satisfies the same recurrence, with the same initial values when $|E(G)|=1$, this provides another way to see that $|S(G)| = |B(G)|$.  
\end{remark}

\begin{remark}
The proof of Corollary~\ref{cor:BernardiBijective}, combined with the results of \cite{ABKS} and \cite{Backman}, provides another `efficient bijective proof'\footnote{By an {\em efficient bijective proof}, we mean (in this context) a bijection between $\Pic^0(G)$ and the set of spanning trees of $G$ which is efficiently computable in both directions.} of Kirchhoff's Matrix-Tree Theorem in the spirit of \cite{Baker-Shokrieh}, as well as a new algorithm for choosing a uniformly random spanning tree of $G$.\footnote{The results of \cite{Backman} can be used to prove that the inverse of the natural map $B(G) \to Pic^g(G)$ is efficiently computable.  See \cite{Baker-Shokrieh} for an explanation of how such a bijection can be used to find random spanning trees.}
\end{remark}

\section{The Bernardi torsor}
\label{sec:Bernardi_torsor}

In this section, we show how to associate a simply transitive action $\beta_v$ of $\Pic^0(G)$ on $S(G)$ to a pair $(G,v)$ consisting of a ribbon graph $G$ and a vertex $v$ of $G$.
This comes down to showing that two Bernardi bijections $\beta_{(v,e)}$ and $\beta_{(v,e')}$ associated to the same root vertex differ merely by translation by some element of 
$\Pic^0(G)$.
For a divisor $D \in B(G)$, we write $[D]$ for the linear equivalence class of $D$ in $\Pic^g(G)$.
Recall that $\Pic^0(G)$ acts simply and transitively on $\Pic^g(G)$ by addition, and that the map $D \mapsto [D]$ gives a canonical bijection from $B(G)$ to $\Pic^g(G)$.   From this we get a canonical simply transitive action of $\Pic^0(G)$ on $B(G)$ sending $D$ to the unique break divisor $\gamma \cdot D$ linearly equivalent to $D + \gamma$.

\begin{theorem} \label{theorem:independent_of_e}
Let $v$ be a vertex of $G$.
\begin{enumerate}
\item Let $e_1,e_2$ be edges incident to $v$, and
let $\beta_1 = \beta_{(v,e_1)}$ and $\beta_2 = \beta_{(v,e_2)}$ be the corresponding Bernardi bijections.
Then there exists an element $\gamma_0 \in \Pic^0(G)$ such that $\beta_2(T) = \gamma_0 \cdot \beta_1(T)$ for all $T \in S(G)$.
\item The action $\beta_v : \Pic^0(G) \times S(G) \to S(G)$ defined by $\gamma \cdot T := \beta_{(v,e)}^{-1} \left( \gamma \cdot \beta_{(v,e)}(T) \right)$ for any edge $e$ incident to $v$, depends only on $v$ and not on the choice of $e$.
\end{enumerate}
\end{theorem}

\begin{proof}
Assuming (1) for the moment, we verify that (2) holds.  We need to prove that if $\beta_1 = \beta_{(v,e_1)}$ and $\beta_2 = \beta_{(v,e_2)}$ are as in (1),
then
\[
\beta_1^{-1} \left( \gamma \cdot \beta_1(T) \right) = \beta_2^{-1} \left( \gamma \cdot \beta_2(T) \right).
\]
To see this, observe that (1) implies $\beta_2^{-1}(x) = \beta_1^{-1}(\gamma_0^{-1} \cdot x)$ for all $x \in \Pic^g(G)$.  Thus
\[
\begin{aligned}
\beta_2^{-1} \left( \gamma \cdot \beta_2(T) \right) &= \beta_1^{-1}\left( \gamma_0^{-1} \gamma \gamma_0 \cdot \beta_1(T) \right) \\
&= \beta_1^{-1} \left( \gamma \cdot \beta_1(T) \right) \\
\end{aligned}
\]
as claimed. (Note that we use in a crucial way the fact that $\Pic^0(G)$ is abelian.)

\medskip

For (1), it suffices to prove that $\beta_1(T)-\beta_2(T)$ and $\beta_1(T')-\beta_2(T')$ are linearly equivalent in $\Div^0(G)$ for any two spanning trees $T,T'$ of $G$.  To do this we first derive a useful formula for $\beta_1(T)-\beta_2(T)$.  By definition, we have
\begin{equation*}
\beta_1(T)-\beta_2(T) = \sum_{f \not\in T} \delta(f),
\end{equation*}
where $\delta(f) := \eta_{(v,e_1)}(f) - \eta_{(v,e_2)}(f)$ (considered as a divisor on $G$).
Thus it will suffice to find a formula for $\delta(f)$ when $f\not\in T$.

\medskip

Let the cyclic ordering of the edges around $v$, starting with $e_1$, be
\[
(e_1, a_1,\ldots, a_k, e_2, b_1,\ldots,b_\ell).
\]
Let $I = \{ e_1,a_1,\ldots,a_k \}$ and $J = \{ e_2,b_1,\ldots,b_\ell \}$.
Removing $v$ from $T$ partitions the set $V(G) \backslash \{ v \}$ into disjoint sets $A$ and $B$, where $A$ (resp. $B$) is the union of all vertices lying in the same connected component of $T \backslash v$ as some edge in $I$ (resp. $J$).  See Figure~\ref{fig:1-3}..

\medskip

The Bernardi tours $\tau_{(v,e_1)}(T)$ and $\tau_{(v,e_2)}(T)$ are cyclic shifts of each other, the difference
being that $\tau_{(v,e_1)}(T)$ traverses the $A$-components of $T$ followed by the $B$-components, while the reverse is true for
$\tau_{(v,e_2)}(T)$.  This shows that $\delta(f) = 0$ (i.e., the Bernardi tours $\tau_{(v,e_1)}(T)$ and $\tau_{(v,e_2)}(T)$ cut through $f$ from the same vertex) when $f \in E(G) \backslash T$ is any one of the following:
\begin{itemize}
  \item A loop edge.
  \item An edge whose endpoints both belong to $A$ or both belong to $B$.
  \item An edge $va \in I$ with $a \in A$, or an edge $vb \in J$ with $b \in B$.
\end{itemize}

On the other hand, the following kind of edges of $G \backslash T$ make a non-trivial contribution to the difference $\delta(f) := \eta_{(v,e_1)}(f) - \eta_{(v,e_2)}(f)$ (considered as a divisor on $G$):

\begin{itemize}
  \item If $f=ab$ with $a \in A$ and $b \in B$ then $\delta(f) = (a) - (b)$.
  \item If $f=va' \in J$ with $a' \in A$ then $\delta(f) = (a') - (v)$.
  \item If $f=vb' \in I$ with $b' \in B$ then $\delta(f) = (v) - (b')$.
\end{itemize}

Summarizing, we obtain the following formula, where the sum is over edges not in $T$:
\begin{equation} \label{eq:betaTdiff}
\beta_1(T)-\beta_2(T) = \sum_{\substack{f=ab \\ a \in A, b \in B}} (a) - (b) + \sum_{\substack{f = va' \in J \\ a' \in A}} (a')-(v) + \sum_{\substack{f=vb' \in I \\ b' \in B}}(v)-(b').
\end{equation}

See Figure~\ref{fig:1-3} for an example.

\begin{figure}
\centering
\includegraphics[width=.45\textwidth]{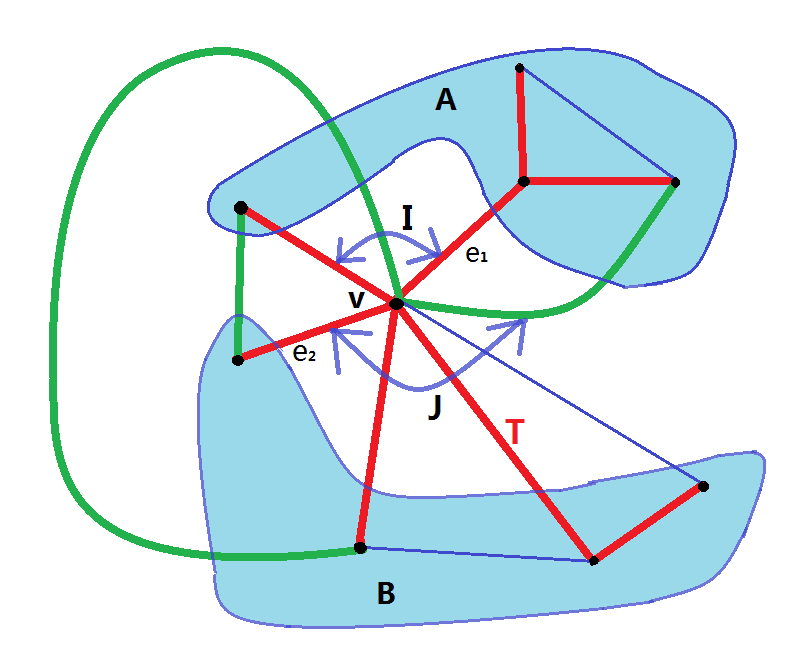}
\\
(a)
\\
\includegraphics[width=.45\textwidth]{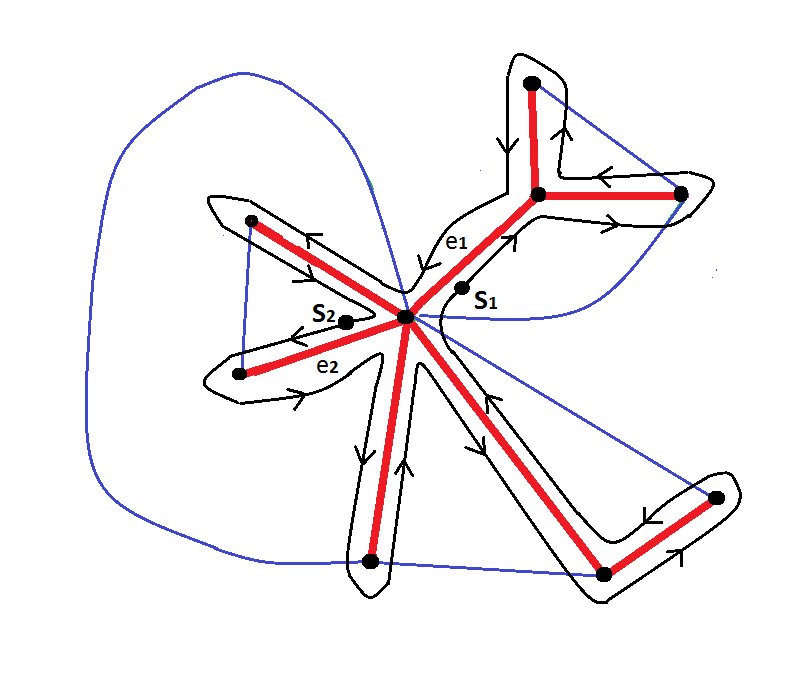}
\\
(b)
\\
\includegraphics[width=.45\textwidth]{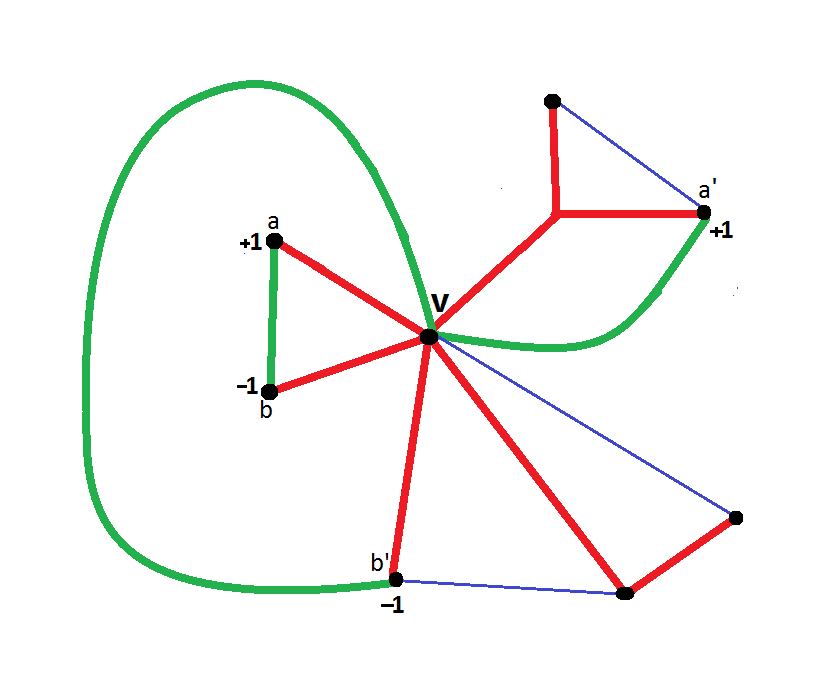}
\\
(c)
\\
\caption{(a) The spanning tree $T$ is shown in red.  The edges which contribute non-trivially to $\beta_1(T)-\beta_2(T)$ are shown in green.  (b) The Bernardi tours $\tau_1 = \tau_{(v,e_1)}(T)$ and $\tau_2 = \tau_{(v,e_2)}(T)$ differ by a cyclic shift: $\tau_1$ begins at $s_1$ and $\tau_2$ begins at $s_2$.  (c) The difference $\beta_1(T) - \beta_2(T)$.}
  \label{fig:1-3}
\end{figure}

\medskip

We now modify the expression in (\ref{eq:betaTdiff}) by firing each vertex in $A$.  More formally, since the characteristic function $\chi_A$ satisfies
\begin{equation*} 
\begin{aligned}
\Delta(\chi_A) &= \sum_{\substack{f=ab \\ a \in A, b \in B}} (a) - (b) + \sum_{\substack{f = va' \\ a' \in A}} (a')-(v) \\
&= \sum_{\substack{f=ab \\ a \in A, b \in B}} (a) - (b) + \sum_{\substack{f = va' \in I \\ a' \in A}} (a')-(v) + \sum_{\substack{f = va' \in J \\ a' \in A}} (a')-(v), \\
\end{aligned}
\end{equation*}

we have
\begin{equation*} 
\beta_1(T)-\beta_2(T) = \Delta(\chi_A) + \sum_{f=vu \in I} (v) - (u),
\end{equation*}
and in particular
\begin{equation} 
\label{eq:betaTdiff2}
\beta_1(T)-\beta_2(T) \sim \sum_{f=vu \in I} (v) - (u).
\end{equation}

Since the right-hand side of (\ref{eq:betaTdiff2}) does not depend on $T$, part (1) of the theorem follows
(with $\gamma_0$ equal to $\sum_{f=vu \in I} (u) - (v)$).
 \end{proof}

\begin{corollary}
If $G$ is a ribbon graph and $v$ is a vertex of $G$, then the action $\beta_v$ defined above makes the set of spanning trees of $G$ into a torsor for
$\Pic^0(G)$.
\end{corollary}

\section{Planarity and the dependence of the Bernardi torsor on the base vertex}
\label{sec:planar_independence}

Given a ribbon graph $G$, we prove that the action $\beta_v$ defined in the previous section is independent of the vertex $v$ if and only if $G$ is {\em planar}.

\medskip

First, we deal with the case where $G$ is planar:

\begin{theorem} \label{theorem:planar_canonical}
If $G$ is a planar ribbon graph, then the action $\beta_v$ is independent of $v$, and hence defines a {\em canonical} action
$\beta$ of $\Pic^0(G)$ on $S(G)$.
\end{theorem}

\begin{proof}

Since $G$ is connected by assumption, it suffices to prove that $\beta_{v_1} = \beta_{v_2}$ whenever $v_2$ is a neighbor of $v_1$.
Without loss of generality, we may assume that the ribbon structure corresponds to the counterclockwise orientation of the plane.

\medskip

Let $e_1$ be an edge connecting $v_1$ and $v_2$,
and let $e_2$ be the edge following $e_1$ in the cyclic ordering around $v_2$.  (If ${\rm deg}(v_2)=1$ then we will have $e_1=e_2$.)
By Theorem~\ref{theorem:independent_of_e}, which allows us to pick whichever edges we want in our initial data, it suffices to prove that for each spanning tree $T$ we have $\beta_{(v_1,e_1)}(T)=\beta_{(v_2,e_2)}(T)$.

\medskip

If $e_1 \in T$, then since the Bernardi tours starting with $(v_1,e_1)$ and $(v_2,e_2)$ are cyclic shifts of each other (they coincide other than the fact that the first tour starts with $e_1$ and the second ends with $e_1$), we have $\beta_{(v_1,e_1)}(T)=\beta_{(v_2,e_2)}(T)$.
We may therefore assume that $e_1\not\in T$.
In this case, $T \cup e_1$ contains a unique simple cycle $C = C_{T,e_1}$, called the {\em fundamental cycle} associated to $T$ and $e_1$.
Since $G$ is planar, the edges other than $e_1$ which are not in the spanning tree $T$ can be partitioned into two disjoint subsets: the edges $E_{\rm in}$ lying {\em inside} $C$ and the edges $E_{\rm out}$ lying {\em outside} $C$.

\medskip

The Bernardi process associated to the initial data $(v_1,e_1)$ will cut through $e_1$, then cut through all of the edges in $E_{\rm in}$, then cut through all
the edges in $E_{\rm out}$, touring around $T$ in the process.
The Bernardi process associated to the initial data $(v_2,e_2)$ will cut through all of the edges in $E_{\rm out}$, then cut through $e_1$, then cut through all
the edges in $E_{\rm in}$, touring around $T$ in the process.
(See Figure~\ref{fig:8-9} for an example.)

\begin{figure}
\centering
\includegraphics[width=.4\textwidth]{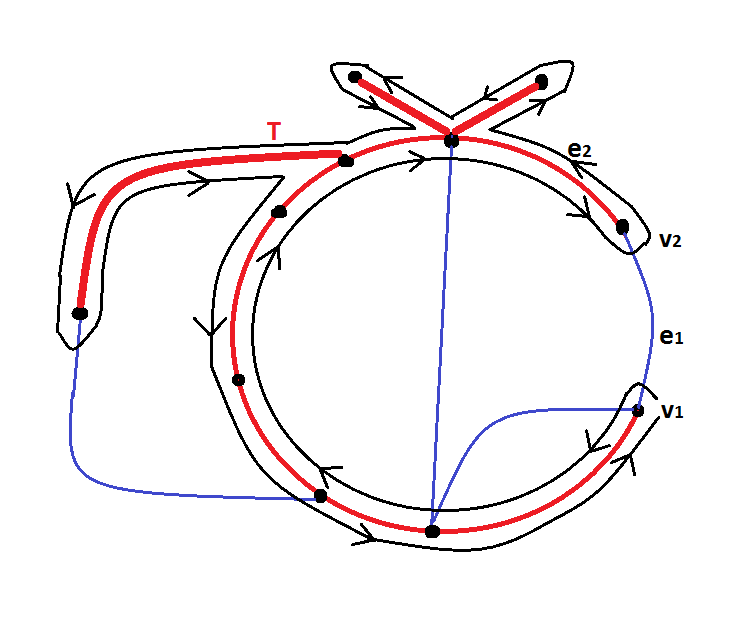}
\includegraphics[width=.4\textwidth]{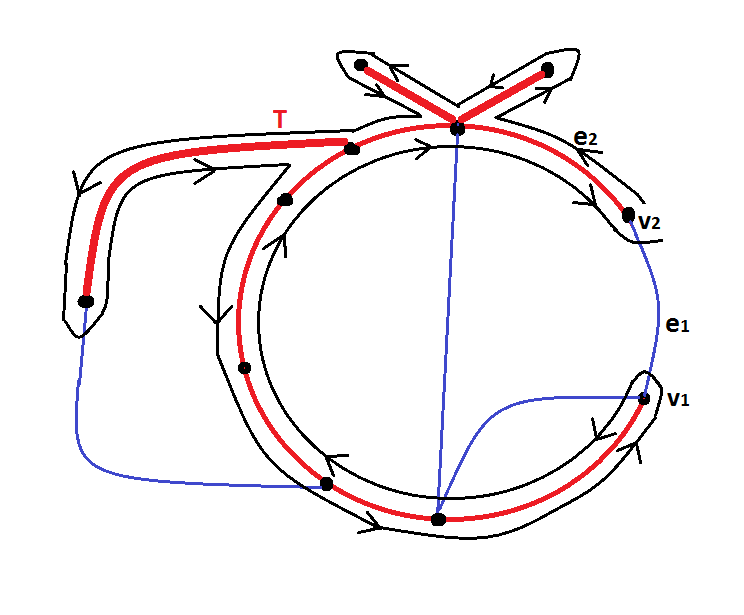}
\caption{The tours associated to $(v_1,e_1)$ and $(v_2,e_2)$, respectively, for the planar graph $G$ and the spanning tree $T$ (in red).}
  \label{fig:8-9}
\end{figure}

It follows that not only are the tours $\tau_{(v_i,e_i)}$ for $i=1,2$ the same up to a cyclic shift, they also cut through edges not in $T$ in exactly the same way.
In particular, $\beta_{(v_1,e_1)}=\beta_{(v_2,e_2)}$.
\end{proof}

\begin{remark}
We conjecture that if $G$ is a planar ribbon graph, the canonical Bernardi bijection between spanning trees of $G$ and break divisors of $G$ is ``geometric'' in the sense of \cite[Remark 4.26]{ABKS}.\footnote{Note added: this conjecture has now been proved by Chi Ho Yuen.}
\end{remark}

Next, we treat the non-planar case.  We begin with a simple lemma:

\begin{lemma}
\label{lemma:cut}
Let $O$ be an acyclic orientation of a connected finite graph $G$, and let $B$ be a non-empty subset of $E(G)$.  Orient each edge in $B$ according to 
$O$.  Let $\partial : C_I \to {\rm Div}^0(G)$ be the natural boundary map, where $C_I$ is the lattice of integer $1$-chains on $G$. 
If the class of $\sum_{\vec{e} \in B} \partial(\vec{e})$ in ${\rm Pic}^0(G)$ is zero, then $B$ is a union of cuts in $G$.
\end{lemma}

\begin{proof}
As mentioned in \S\ref{sec:planar_duality_background}, the map $\partial$ induces an isomorphism 
\[
\frac{C_I}{Z_I \oplus B_I} \isomap {\rm Pic}^0(G).
\]
Thus we can write $c := \sum_{\vec{e} \in B} \vec{e} \in C_I$ as a sum $z + b$ with $z \in Z_I$ and $b \in B_I$.
As the orientation $O$ is acyclic, we must have $z=0$.  Therefore $c=b \in B_I$ is a sum of directed cuts, and in particular $B$ is a union of cuts.
\end{proof}

\begin{theorem} \label{theorem:non-planar_non-canonical}
If $G$ is a non-planar ribbon graph, there are vertices $v,v'$ of $G$ with $\beta_v \neq \beta_{v'}$.
\end{theorem}

\begin{proof}
By the discussion in Section~\ref{sec:ribbon_graphs}, $G$ has a polygonal representation inside a fundamental polygon \[
P=a_1 b_1 a_1^{-1} b_1^{-1}\cdots a_g b_g a_g^{-1} b_g^{-1}
\]
which we may assume to have {\em minimal genus} among all such representations.
We may also assume that the drawing of $G$ inside $P$ has the minimum possible number of edges passing through the polygon $P$.

\medskip

Let $a = a_1, b = b_1$.
Since $G$ is non-planar, we may assume that there are edges $e$ and $e'$ of $G$ which intersect boundary edges $a,a^{-1}$ and $b,b^{-1}$ of $P$, respectively.

\medskip

Let $G_0$ be the complement in $G$ of all edges which pass through $a$ or $b$.
Then $G_0$ is connected, since otherwise one could redraw $G$ by changing the relative position of the components of $G_0$ and
obtain a polygonal representation that contradicts the minimality of $G$ and $P$.

\medskip

Since $G_0$ is connected, there exists a spanning tree $T_1$ of $G$ contained in $G_0$.
Let $C = C_{T_1,e}$, let $e^\star$ be an edge of $T_1 \cap C$ (so in particular $e^\star$ does not intersect $a$ or $b$), and let $T_2 = T_1 \cup e \backslash e^\star$.
Without loss of generality, we may orient $C$ and label the endpoints of $e,e^\star$ so that $\vec{e} = (x,y)$ and $(\vec{e})^\star = (x^\star,y^\star)$ are oriented consistently in $C$. Let $e^{\star\star}$ be the edge following $e^\star$ in the cyclic orientation around $y^\star$ and
consider the Bernardi maps $\beta_1$ and $\beta_2$ arising from the initial data $(x^\star,e^\star)$ and $(y^\star,e^{\star\star})$, respectively.

\medskip

Since $e^\star \in T_1$, we know that
\[
\beta_1(T_1)-\beta_2(T_1) = 0.
\]

On the other hand, the Bernardi tours $\tau_{(x^\star,e^\star)}(T_2)$ and $\tau_{(y^\star,e^{\star\star})}(T_2)$ have the property that for each edge $e'$ not in $T_2$,
$\eta_{(x^\star,e^\star)}(e') = \eta_{(y^\star,e^{\star\star})}(e')$ if and only if $e'$ does not pass through $b$ and $b^{-1}$.  Thus
\[
\beta_1(T_2)-\beta_2(T_2) = \sum_{e \in B} \partial e,
\]
where $B$ is the (non-empty) set of edges of $G$ passing through $b$ and $b^{-1}$, oriented so that the head of each edge in $B$ lies on the
path from $x^\star$ to $y^\star$ in the Bernardi tour of spanning tree $T_2$ with initial data $(x^\star,e^\star)$.
(See Figure~\ref{fig:T1-T2} for an example.)

\medskip

Suppose for the sake of contradiction that the divisor $\sum_{e \in B} \partial e$ is linearly equivalent to $0$.
By Lemma~\ref{lemma:cut}, $B$ is a union of cuts.  Thus there is a non-empty connected subgraph $H$ of $G$ with the property that
every edge of $G$ connecting $H$ to its complement is contained in $B$, and in particular passes through $b$ and $b^{-1}$.
But in this case, we can redraw the embedding of $G$ by moving $H$, obtaining a polygonal representation which contradicts the minimality of $G$ and $P$.
\end{proof}

\begin{figure}
\centering
\includegraphics[width=.4\textwidth]{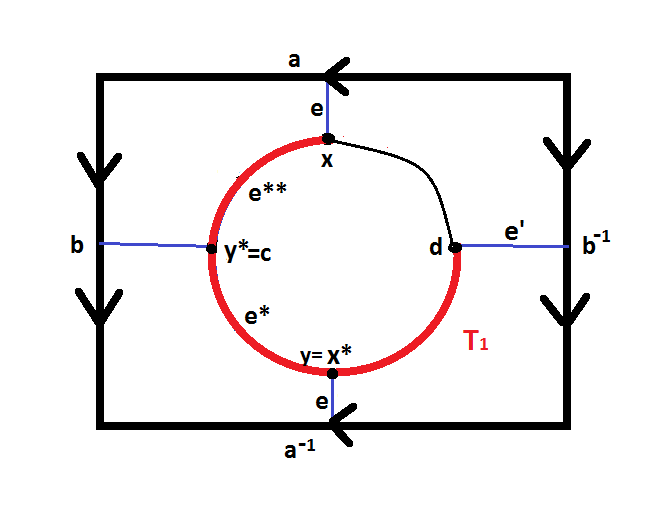}
\includegraphics[width=.4\textwidth]{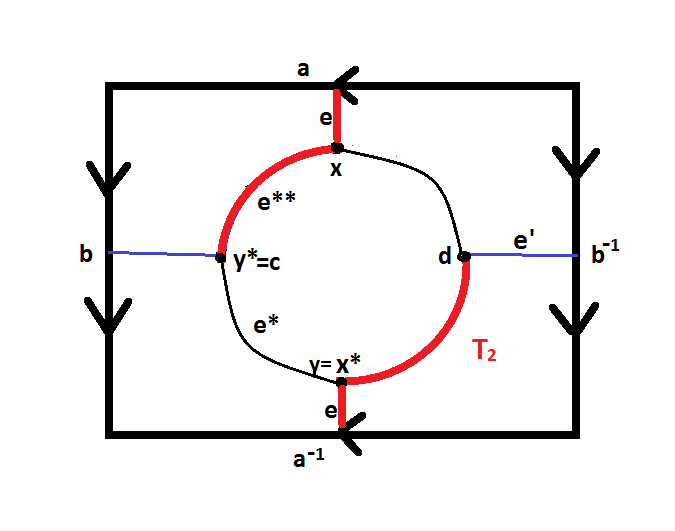}
\caption{An example illustrating the proof of Theorem~\ref{theorem:non-planar_non-canonical}.
In this example, $\beta_1(T_2)-\beta_2(T_2) = (c) - (d)$.}  
\label{fig:T1-T2}
\end{figure}

\section{Compatibility of the Bernardi torsor with planar duality}
\label{sec:planar_duality}

If $G$ is a planar ribbon graph, we show that the natural action of $\Pic^0(G)$ on $S(G)$ is compatible with planar duality:

\begin{theorem} \label{theorem:planar_duality}
Let $G$ be a planar ribbon graph.  Then the natural actions of $\Pic^0(G)$ on $S(G)$ and of $\Pic^0(G^\star)$ on $S(G^\star)$ defined by the Bernardi process are identified with one another via the canonical isomorphism $\Psi : \Pic^0(G) \isomap \Pic^0(G^\star)$ and the canonical bijection $\sigma : S(G) \to S(G^\star)$ defined in Section~\ref{sec:planar_duality_background}.
In other words, the following diagram is commutative:
\[
\begin{CD}
\Pic^0(G) \times S(G) @>>> S(G) \\
@VV{\Psi \times \sigma}V	@VV{\sigma}V \\
\Pic^0(G^\star) \times S(G^\star) @>>> S(G^\star) \\
\end{CD}
\]
\end{theorem}

\begin{proof}
Fix some arbitrary initial data $(v,e)$ and $(v^\star,e^\star)$ for the Bernardi processes on $G$ and $G^\star$, and denote by $\beta : S(G) \to B(G)$ and $\beta^\star : S(G^\star) \to B(G^\star)$ the corresponding Bernardi maps.
We need to prove that for any spanning trees $T_1$ and $T_2$ in $S(G)$ and their corresponding dual spanning trees $T_1^\star$ and $T_2^\star$ in
$S(G^\star)$, we have
\begin{equation} \label{eq:commdiag}
\Psi([\beta(T_2)-\beta(T_1)]) = [\beta^\star(T_2^\star)-\beta^\star(T_1^\star)].
\end{equation}

Without loss of generality, we may assume that $T_2$ is obtained from $T_1$ by adding an edge $e_1 \not\in T_1$ and deleting an edge $e_2 \in T_1$ from the fundamental cycle $C = C_{T_1,e_1}$.  One checks easily that $T_2^\star$ is obtained from $T_1^\star$ by adding $e_2^\star$ and deleting $e_1^\star$.

\medskip

We may also assume without loss of generality that the ribbon structure on $G$
corresponds to the counterclockwise orientation on the plane.  
Orient $e_1$ according to how it is first traversed by the Bernardi tour $\tau(T_2)$, and
orient $e_2$ according to how it is first traversed by the Bernardi tour $\tau(T_1)$.
Let $\vec{e}_1=(x_1,y_1)$ and
$\vec{e}_2=(x_2,y_2)$ be the resulting oriented edges.

\medskip

By Theorem~\ref{theorem:planar_canonical}, we may assume without loss of generality that the initial data for the Bernardi process on $G$ are
$(v,e) = (x_1,e_1)$.

\medskip

Deleting the edge $e_2$ from $T_1$ (or, alternatively, deleting the edge $e_1$ from $T_2$) defines a partition of $V(G)$ into disjoint subsets $A$ and $B$ with $y_1,y_2 \in A$ and $x_1,x_2 \in B$.

\medskip

Let $F$ be the set of oriented edges of $G$ not belonging to $T_1 \cup T_2$ which connect vertices in $B$ to vertices in $A$ and lie
on the {\em inside} of the cycle $C$. 

\medskip

We claim that
\begin{equation} \label{eq:dual_difference}
\beta(T_2)-\beta(T_1)=(y_2)-(x_1) + \sum_{\vec{f} \in F} \partial(\vec{f}).
\end{equation}

\begin{figure}
\centering
\includegraphics[width=.4\textwidth]{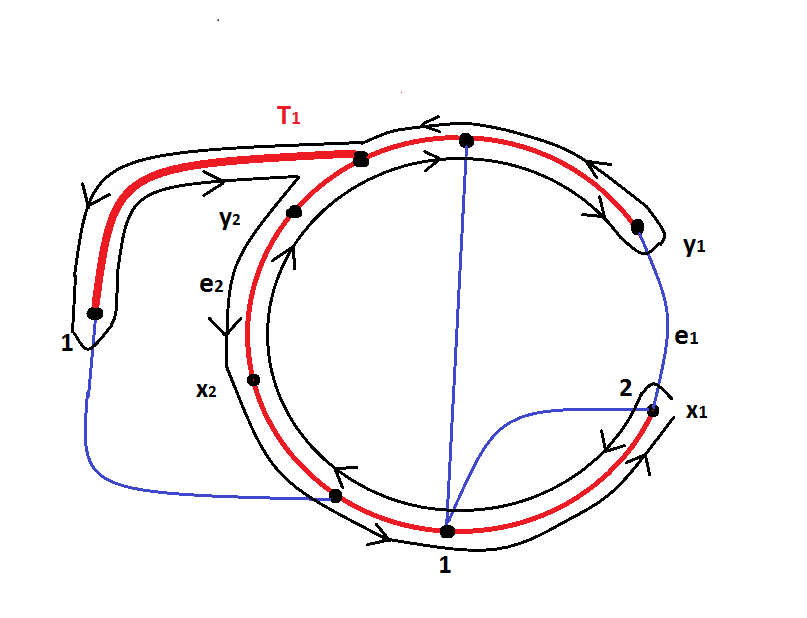}
\\
(a)
\\
\includegraphics[width=.4\textwidth]{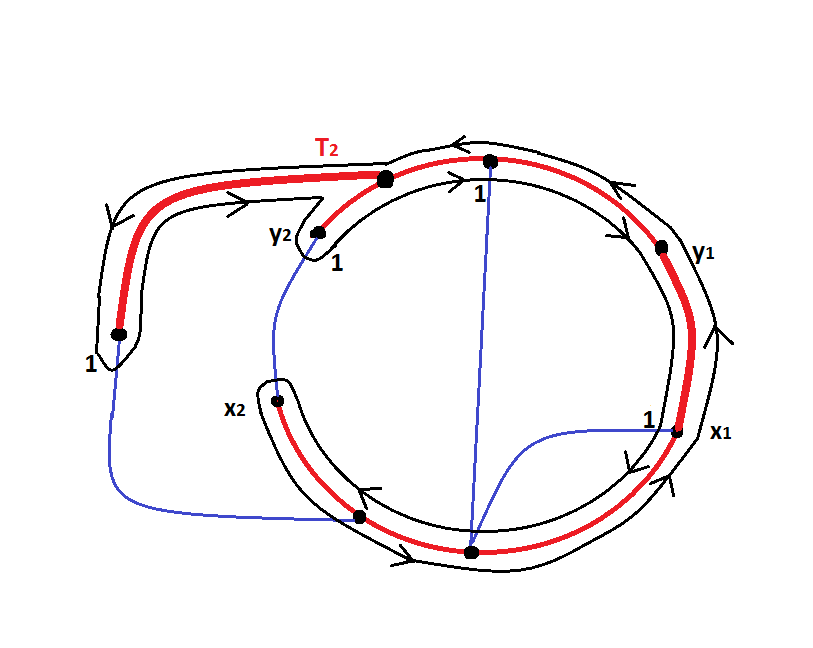}
\\
(b)
\\
\includegraphics[width=.4\textwidth]{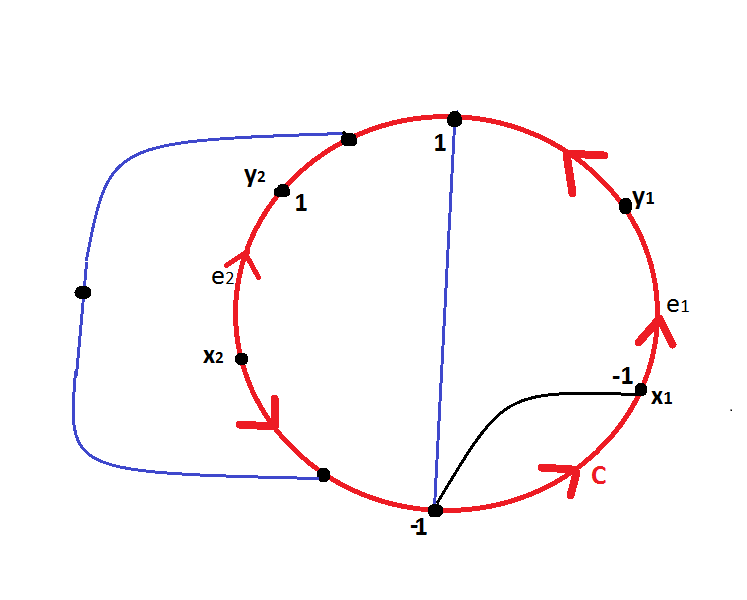}
\\
(c)
\\
\caption{(a),(b): The Bernardi tours $\tau_{(x_1,e_1)}(T_1)$ and $\tau_{(x_1,e_1)}(T_2)$ associated to two
different spanning trees (shown in red), and their associated break divisors. (c): The difference $\beta_{(x_1,e_1)}(T_2)-\beta_{(x_1,e_1)}(T_1)$ between the break divisors associated to $T_2$ and $T_1$.  The fundamental cycle $C = C_{T_1,e_1}$ is shown in red.}
\label{fig:10-11}
\end{figure}

Indeed, we can write the tour $\tau_{(x_1,e_1)}(T_1)$ as 
\[
\tau_{(x_1,e_1)}(T_1) = (\alpha_1,\vec{e}_2,\alpha_2,\alpha_3,(\vec{e_2})^{\rm op},\alpha_4),
\]
where $\alpha_1$ goes from $x_1$ to
$x_2$ inside $C$, $\alpha_2$ goes from $y_2$ to $y_1$ inside $C$, $\alpha_3$ goes from $y_1$ to $y_2$ outside $C$, and $\alpha_4$ goes from $x_2$ to $x_1$ outside $C$.  

Similarly, we can write the tour $\tau_{(x_1,e_1)}(T_2)$ as 
\[
\tau_{(x_1,e_1)}(T_2) =  (\vec{e}_1,\alpha_3,\alpha_2,(\vec{e_1})^{\rm op},\alpha_1,\alpha_4).
\] 
The desired formula (\ref{eq:dual_difference}) follows easily.  
The point here is that there are no edges joining ``inside'' to ``outside'' vertices, and the ``outside-to-outside'' edges, which are cut during $\alpha_3 \cup \alpha_4$, are traversed in the same order in the Bernardi tours associated to $T_1$ and $T_2$.  Thus the difference $\beta(T_2)-\beta(T_1)$ comes from the ``inside-to-inside'' edges joining $A$ (the set of vertices encountered by $\alpha_2$ and $\alpha_3$) and $B$ (the set of vertices encountered by $\alpha_1$ and $\alpha_4$), which are cut during $\alpha_1 \cup \alpha_2$ and are traversed in the opposite order in the Bernardi tours associated to $T_1$ and $T_2$.  See Figure~\ref{fig:10-11} for an example. 

\medskip

We can perform a similar calculation on the dual side.  In this case, recall that the ribbon structure on $G^\star$
corresponds to the {\em clockwise} orientation on the plane.  
Orient $e_1^\star$ according to how it is first traversed by the Bernardi tour $\tau(T_1^\star)$, and
orient $e_2^\star$ according to how it is first traversed by the Bernardi tour $\tau(T_2^\star)$.
Let $(\vec{e}_1)^\star=(x_1^\star,y_1^\star)$ and
$(\vec{e}_2)^\star=(x_2^\star,y_2^\star)$ be the resulting oriented edges.
Let $C^\star$ be the fundamental cycle for $e_2^\star$ with respect to $T_1^\star$, which coincides with the fundamental cycle for 
$e_1^\star$ with respect to $T_2^\star$.

\medskip

By Theorem~\ref{theorem:planar_canonical}, we may assume without loss of generality that the initial data for the Bernardi process on $G^\star$ are
$(v^\star,e^\star) = (x_1^\star,e_1^\star)$.
Deleting the edge $e_1^\star$ from $T_1^\star$ (or, alternatively, deleting the edge $e_2^\star$ from $T_2^\star$) defines a partition of $V(G^\star)$ into disjoint subsets $A^\star$ and $B^\star$ with $y_1^\star,y_2^\star \in A^\star$ and $x_1^\star,x_2^\star \in B^\star$.

\medskip

Let $F^\star$ be the set of oriented edges of $G^\star$ not belonging to $T_1^\star \cup T_2^\star$ which connect vertices in $B^\star$ to vertices in 
$A^\star$ and lie on the {\em outside} of the cycle $C^\star$.

\medskip

We claim that
\begin{equation} \label{eq:dual_difference_2}
\beta(T_1^\star)-\beta(T_2^\star)=(y_2^\star)-(x_1^\star) + \sum_{\vec{f}\in F^\star} \partial(\vec{f}).
\end{equation}

The proof is similar to the previous argument, with $T_1^\star$ playing the role of $T_2$ and $T_2^\star$ playing the role of $T_1$.
Also, since the Bernardi tours now go clockwise, the edges {\em outside} $C^\star$ play the role previously played by the edges inside $C$.
See Figure~\ref{fig:6} for an illustration of the situation. 

\begin{figure}
\centering
\includegraphics[width=.6\textwidth]{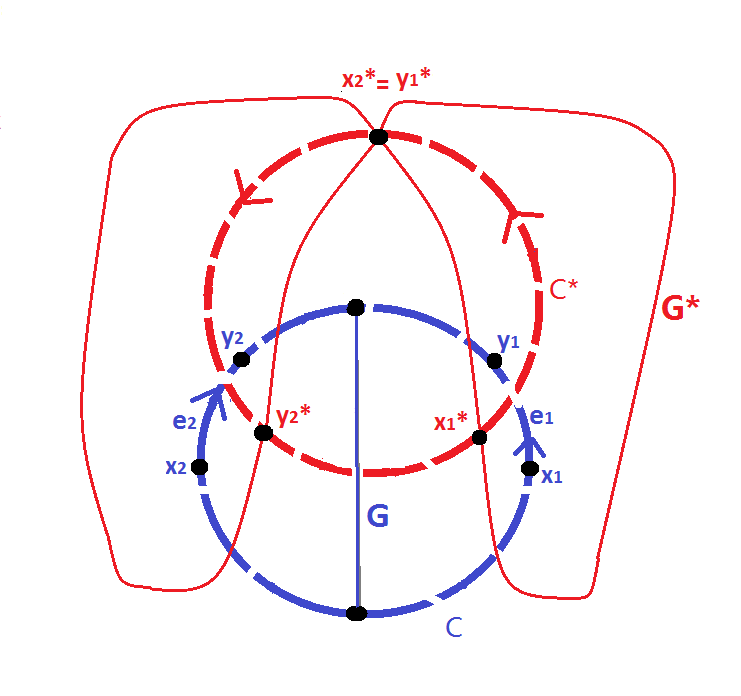}
\caption{An illustration of the proof of Theorem~\ref{theorem:planar_duality}.}
\label{fig:6}
\end{figure}

\medskip

Define $\lambda = \vec{e}_2 + \mu+\nu \in C_I$, where $\mu$ is the sum of all the oriented edges in the {\em clockwise} path along $C$ from $x_1$ to $x_2$ and $\nu$ is the sum of all the oriented edges in $F$.  By (\ref{eq:dual_difference}), we have $\partial(\lambda)=\beta(T_2)-\beta(T_1)$, where $\partial: C_I \to \Div^0(G)$ is as in \S\ref{sec:planar_duality_background}.

\medskip

Similarly, define $\lambda^\star = ((\vec{e}_2)^{\star})^{\rm op} + \mu^\star + \nu^\star \in C_I^\star$, where $\mu^\star$ is the sum of all the oriented edges in the {\em counterclockwise} path along $C^\star$ from $x_2^\star$ to $x_1^\star$ and $\nu^\star$ is the sum of all the oriented edges in $(F^\star)^{\rm op}$.  By (\ref{eq:dual_difference_2}), we have $\partial(\lambda^\star)=\beta(T_2^\star)-\beta(T_1^\star)$.

\medskip

One now checks that, under the natural duality isomorphism $\psi : C_I \to C_I^\star$ defined in Section~\ref{sec:planar_duality_background},
$\psi$ takes $\vec{e}_2$ to $((\vec{e}_2)^{\star})^{\rm op}$, $\mu$ to $\nu^\star$, and $\nu$ to $\mu^\star$.
Thus $\psi$ takes $\partial(\lambda)$ to $\partial(\lambda^\star)$, which means that $\Psi$ takes
$\beta(T_2)-\beta(T_1)$ to $\beta(T_2^\star)-\beta(T_1^\star)$.  This establishes (\ref{eq:commdiag}).
\end{proof}

\begin{remark}
The first author conjectured the analogue of Theorem~\ref{theorem:planar_duality} for the rotor-routing process at an AIM workshop in July 2013.
Together with Theorem~\ref{theorem:bernardi_rotor-routing_comparison} in the next section, Theorem~\ref{theorem:planar_duality} affirms our conjecture.  Chan et. al. \cite{Chan-et-al} have independently proved the compatibility of the rotor-routing torsor with planar duality using a different (more direct) method.
\end{remark}

\section{Comparison between the Bernardi and rotor-routing torsors}
\label{sec:comparison}

We show that given a ribbon graph $G$ and a vertex $v$ of $G$, the Bernardi torsor $\beta_v$ defined in this paper and the rotor-routing torsor defined in \cite{Holroyd-et-al} and \cite{CCG} are equal when $G$ is planar.  In particular, the canonical torsor structures on $S(G)$ defined by the Bernardi and rotor-routing processes are {\em the same} for planar ribbon graphs.  We also give an example which shows that $\beta_v$ can be different from $r_v$ when $G$ is non-planar.

\subsection{The planar case}

\begin{theorem} \label{theorem:bernardi_rotor-routing_comparison}
Let $G$ be a planar ribbon graph.  Then the Bernardi and rotor-routing processes define the same $\Pic^0(G)$-torsor structure on $S(G)$.
\end{theorem}

\begin{proof}
Let $\beta$ be the Bernardi bijection associated to some initial data $(v,e)$, and let $T$ be a spanning tree of $G$.
Since $\Pic^0(G)$ is generated by the linear equivalence classes of divisors of the form $(x)-(y)$, with $x,y \in V(G)$, it suffices to prove that
if $T' = \left( (x)-(y) \right)_y T$ then $\beta(T')-\beta(T) \sim (x)-(y)$ for all $x,y \in V(G)$.

\medskip

By Theorem~\ref{thm: planar_rotor-routing}, we may assume that the root vertex for the rotor-routing process is $y$.
Let $T''$ be the first spanning tree after $T$ which appears during the rotor-routing process $\left( (x)-(y) \right)_y$ from $T$ to $T'$, and let
$x''$ be the vertex to which the chip is sent when we reach $T''$.
By induction on the number of rotor-routing steps, it is enough to show that $\beta(T'')-\beta(T) \sim (x)-(x'')$.

\medskip

{\bf Case 1:} $T''$ is obtained from $T$ in just one step of rotor-routing.

\medskip

In this case, $T''$ is obtained from $T$ by deleting an edge $e'$ incident to $x$ and adding an edge $e''$ from $x$ to $x''$.
By Theorems~\ref{theorem:independent_of_e} and \ref{theorem:planar_canonical}, we may assume without loss of generality that the initial data for the Bernardi process are $(x'',e'')$.  
If $L$ denotes the complement in $E(G)$ of $T \cup T''$, the Bernardi tours associated to $T$ and $T''$ will cut edges in $L$ at the same endpoint.  
The difference between $\beta(T)$ and $\beta(T'')$ therefore arises from the fact that the tour associated to $T$ cuts through $e''$ but not $e'$ and the tour associated to $T''$ cuts through $e'$ but not $e''$.
One verifies in this way that
\[
\beta(T'')-\beta(T) = (x) - (x'').
\]

\medskip

{\bf Case 2:} $T''$ is obtained from $T$ in more than one step of rotor-routing.  (See Figure~\ref{fig:RR_reverse_the_cycle} for an example.)

\medskip

In this case (referring back to the notation from Section~\ref{sec:rotor-routing}), if $S_0 = T$ then $S_1$ will consist of 2 connected components $A$ and $B$
such that $A$ contains a unique directed cycle $C$ with $x \in C$ and $B$ contains no directed cycle.
Since $G$ is planar, $C$ is reversible.  Consider the rotor-routing process which takes $(\rho_{\ell_1},x_{\ell_1}) := (\rho_1,x_1)$ to
$(\rho_{\ell_2},x_{\ell_2}) = (\bar{\rho_1},x_1)$.
By \cite[Proposition 6]{CCG}, the set $L_C$ of vertices $v \not\in C$ which are visited by this reversal process depends only on $C$ and is contained in $A$.  
More precisely, in the course of sending the chip back to the initial vertex $x_1$, the reversal process reverses all the directed edges inside or on the cycle $C$ and keeps the rest of the rotor configuration the same.

\medskip

In the next step of rotor-routing, the chip will be sent to $x_{\ell _2+ 1}$.  If $x_{\ell_2 + 1} \not\in A$ then $S_{\ell_2 + 1} = T''$ is a spanning tree.
Otherwise, $S_{\ell_2 + 1}$ will again contain a unique directed cycle $C'$ and the next several steps of rotor-routing will reverse this directed cycle.
This process will continue a finite number of times until we reach some $\ell_t$ such that $x_{\ell_t + 1} = x''$ and $S_{\ell_t + 1} = T''$.

\medskip

It follows that, during the entire process of rotor-routing from $T$ to $T''$, $T''$ can be obtained from $T$ by deleting an edge $e'$ incident to $x$
and the component $B$ and adding an edge $e''$ from $x_{\ell_t}$ to $x''$. Since $T''$ is the first spanning which tree appears in this process, we know that $e'$ is the next edge after $e''$ in the cyclic order  around $T|_A$ joining $A$ to $B$, otherwise the next edge (if it existed) after $e''$ would produce a spanning tree when it is encountered during the rotor routing process. 
As in Case 1, we may assume that the initial data for the Bernardi process are $(x'',e'')$, and by the same argument as Case 1 one then checks that
\[
\beta(T'')-\beta(T) = (x) - (x'')
\]
as desired.
\end{proof}

\begin{figure}
\centering
\includegraphics[width=.5\textwidth]{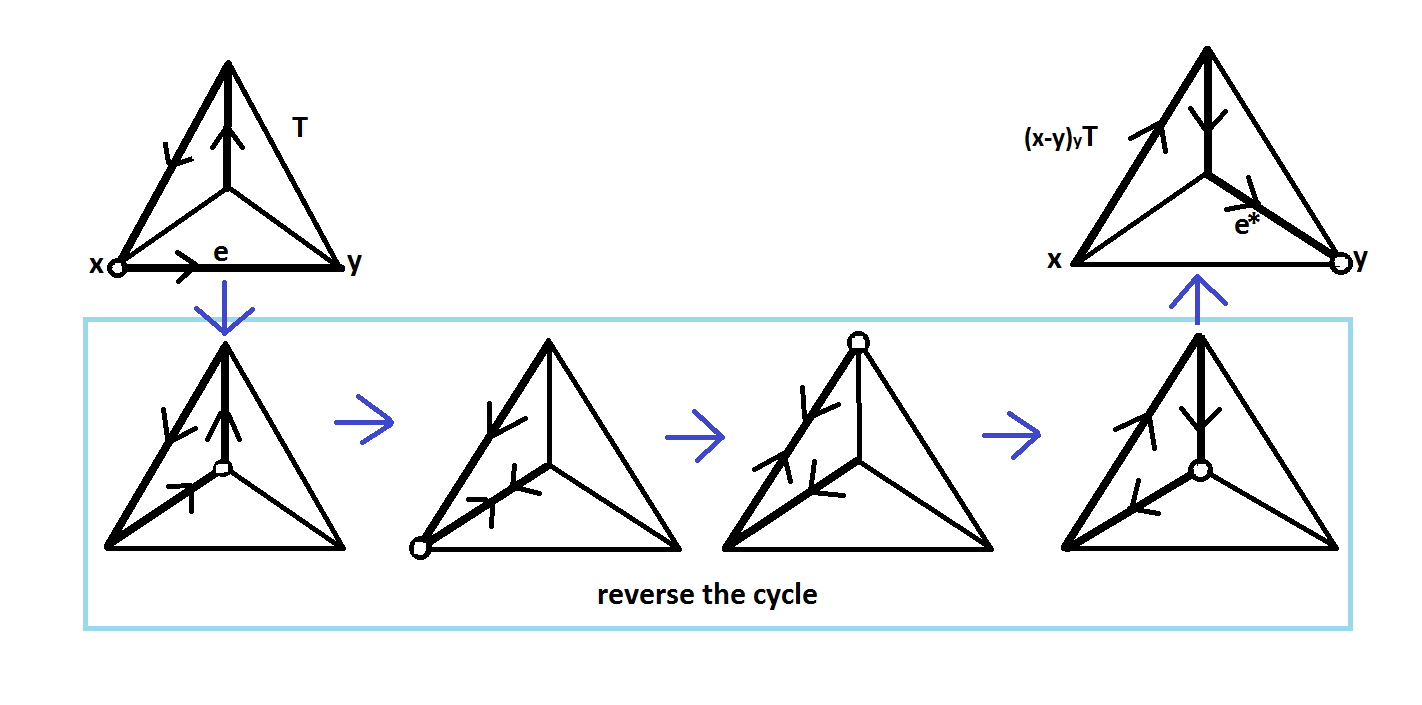}
\caption{An illustration of the proof of Theorem~\ref{theorem:bernardi_rotor-routing_comparison}.}
\label{fig:RR_reverse_the_cycle}
\end{figure}

\subsection{The non-planar case}
We now give an example which shows that for non-planar ribbon graphs, the torsors $\beta_v$ and $\r_v$ can be different.
In Figure~\ref{fig:T-Tstar}, the sink vertex is $x''$ and the chip begins at $x$.  After
one step of rotor-routing, the chip is sent to $x''$ and the spanning tree $T$ is transformed into $T''$.

\medskip

If $\beta_{x''} = r_{x''}$, then we would have
$$\beta_{x''}(T'')-\beta_{x''}(T)\thicksim(x)-(x'').$$

\medskip

However, setting the initial data for the Bernardi process as $(x'', (x'',x))$, we find that
$$\left( \beta_{x''}(T'')-\beta_{x''}(T) \right) - \left( (x)-(x'') \right) = (z)-(y)$$
which is not linearly equivalent to 0.

\begin{figure}
\centering
\includegraphics[width=.4\textwidth]{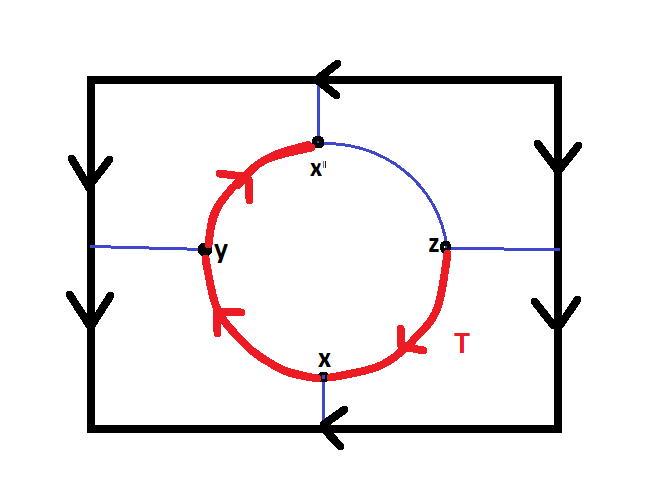}
\includegraphics[width=.4\textwidth]{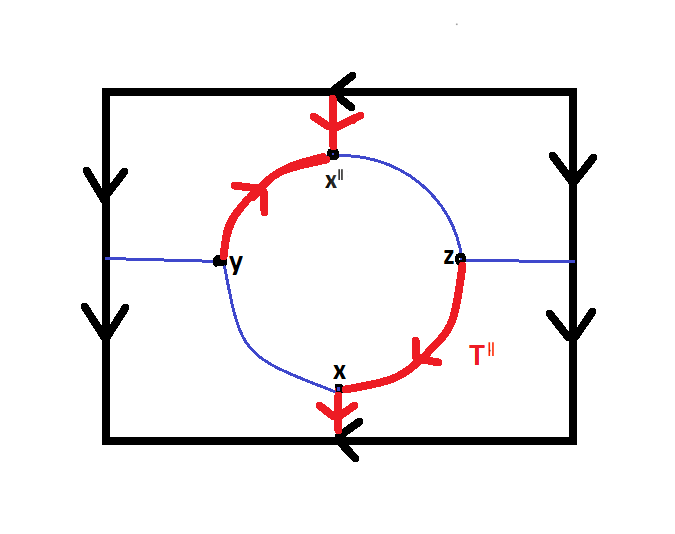}
\caption{Comparison of $\beta_v$ and $r_v$ for a non-planar ribbon graph.}
\label{fig:T-Tstar}
\end{figure}

\medskip

We conclude this paper with the following conjecture, one direction of which is Theorem~\ref{theorem:bernardi_rotor-routing_comparison}.

\begin{conjecture}
Let $G$ be a ribbon graph without loops or multiple edges. The Bernardi and rotor-routing torsors $\beta_v$ and $r_v$ agree for all $v$ if and only if $G$ is planar.
\end{conjecture}

The conjecture holds in numerous examples which we computed by hand.

\bigskip\bigskip


\bibliographystyle{alpha}
\bibliography{Bernardi}
\end{document}